\documentclass[english]{article}
\usepackage[T1]{fontenc}
\usepackage[latin9]{inputenc}
\usepackage{xcolor}
\usepackage{pdfcolmk}
\usepackage{amsmath}
\usepackage{amsthm}
\usepackage{amssymb}
\usepackage{esint}
\PassOptionsToPackage{normalem}{ulem}
\usepackage{ulem}

\makeatletter

\providecolor{lyxadded}{rgb}{0,0,1}
\providecolor{lyxdeleted}{rgb}{1,0,0}

\theoremstyle{plain}
\newtheorem{thm}{\protect\theoremname}
  \theoremstyle{definition}
  \newtheorem{defn}[thm]{\protect\definitionname}
  \theoremstyle{plain}
  \newtheorem{prop}[thm]{\protect\propositionname}
  \theoremstyle{plain}
  \newtheorem{cor}[thm]{\protect\corollaryname}
  \theoremstyle{remark}
  \newtheorem{rem}[thm]{\protect\remarkname}
  \theoremstyle{plain}
  \newtheorem{lem}[thm]{\protect\lemmaname}
  \theoremstyle{definition}
  \newtheorem{example}[thm]{\protect\examplename}

\makeatother

\usepackage{babel}
  \providecommand{\corollaryname}{Corollary}
  \providecommand{\definitionname}{Definition}
  \providecommand{\examplename}{Example}
  \providecommand{\lemmaname}{Lemma}
  \providecommand{\propositionname}{Proposition}
  \providecommand{\remarkname}{Remark}
\providecommand{\theoremname}{Theorem}

\begin{document}

\title{Analysis of the one dimensional inhomogeneous Jellium model with
the Birkhoff-Hopf Theorem }

\author{Raphael Ducatez}
\maketitle
\begin{abstract}
We use the Hilbert distance on cones and the Birkhoff-Hopf Theorem
to prove decay of correlation, analyticity of the free energy and
a central limit theorem in the one dimensional Jellium model with
non constant density charge background, both in the classical and
quantum cases. 
\end{abstract}
\tableofcontents{}

\section{Introduction}

The Jellium model describes a system of electrons interacting with
each other in a continuous background of opposite charge. It is a
very fundamental system in quantum chemistry and condensed matter
physics \cite{GiuVig-05,ParYan-94,brack1993physics}. The model has
been initially introduced by Wigner \cite{wigner1934interaction}.
In (quasi-)one dimension it has then been rigorously studied when
the background is uniform by Kunz \cite{KUNZ1974303}, Brascamp-Lieb
\cite{brascamp2002some}, Aizenman-Martin \cite{aizenman1980structure}
and many others \cite{blanc2015crystallization,Baxter-63,Choquard-75,GruLugMar-78,JanLieSei-09,AizJanJun-10,JanJun-14,lewin2017statistical}.
This model is known to reveal a symmetry breaking called the Wigner
crystal. One major difficulty is that the Coulomb potential is long
range. In dimension one, the interaction is like $-|x-y|$ and therefore
the force between two particles does not depend on their mutual distance
which simplifies a lot the problem. 

In this paper we study the inhomogeneous Jellium model in which the
background is not constant. The inhomogeneous case is very important
for applications, at least in three dimensions \cite{GruLebMar-81,HohKoh-64,LunMar-83}.
The Wigner crystal still appears for a periodic background, provided
that the charge in one period is equal to the charge of the particles.
Here we consider any background in one dimension and the system will
not necessarily be crystallized.

One of the most important properties of the constant background model
is that the classical partition function for $N$ particles can be
written in the form
\[
\mathcal{Z}_{N}(\beta)=\langle a,T(\beta)_{\text{ }}^{N}b\rangle
\]
where $T(\beta)$ is a compact operator with positive kernel in some
$L^{2}$ space, which depends smoothly on the inverse temperature
$\beta$. By the Krein-Rutmann Theorem \cite{amann1976fixed}, $T(\beta)$
has a unique largest eigenvalue $\lambda(\beta)>0$ which is always
non degenerate, hence is also a smooth function of $\beta$. As a
consequence, the free energy per particle behaves as 
\[
f_{N}(\beta)=-\frac{1}{N\beta}\log(\mathcal{Z}_{N}(\beta))=-\frac{1}{\beta}\log(\lambda(\beta))-\frac{1}{\beta N}\log(\langle a,v\rangle\langle v,b\rangle)+O(\kappa^{N})
\]
for large $N$, where $v$ is the unique positive eigenvector associated
with $\lambda(\beta)$ and $\kappa<1$. In fact, $T(\beta)^{N}/\lambda(\beta)^{N}$
is close to the rank-one projection on $v$ and this can also be used
to prove the decay of correlations. 

In the inhomogeneous case, the classical partition function takes
the form
\begin{equation}
\mathcal{Z}_{N}(\beta)=\left\langle a,\prod_{0\leq i\leq N-1}T_{i}(\beta)b\right\rangle \label{eq:FonctionDePartition-1-1}
\end{equation}
where the transitive operators $T_{i}(\beta)$ are no longer equal
to each other. Our goal is to generalize the results proved in the
homogeneous case to the inhomogeneous case. For this we will replace
the spectral approach based on the Krein Rutmann theorem by the Birkhoff-Hopf
theorem \cite{birkhoff1957extensions,hopf1963inequality}. The main
idea behind this method is to quantify how a product of many operators
with positive kernels can be well approximated by a rank-one operator.
A main tool is the so-called Hilbert distance on cones, a concept
which will be discussed at length later on.

Using these tools we will prove the decay of correlations and the
smoothness of the free energy in the inhomogeneous Jellium model.
In the classical case, we can essentially handle any background, but
in the quantum case we require it to be close to a constant. Our method
is general and can be applied to other one dimensional inhomogeneous
systems in statistical physics like the Ising model. It could also
be useful for log gases \cite{ErdosUniversality,forrester2010log}.
For this reason, we will present the theory in the abstract framework
of cones on any Banach spaces, in a form which is well suited to the
setting of statistical physics. 

Our paper is organized as follows. We first describe in Section \ref{sec:Classique}
the Jellium model and state our main results both for the classical
and the quantum cases. We then introduce in Section \ref{sub:Framework-and-Birhoff}
the framework required to state the Birkhoff-Hopf Theorem. Afterward,
we suggest a new formulation of weak ergodicity using rank-one operators
and prove it in Section \ref{sec:Rank-one-operator}. As it is shown
in Section \ref{sec:Decay-of-correlation}, the rank-one approximation
implies that the $k$-particle marginals are well approximated by
(independent) products of the 1-particle marginals. In Section \ref{sec:Smoothness-of-the}
we prove the regularity of the abstract free energy. Finally we deal
with the inhomogeneous Jellium model. Section \ref{sub:proofClass}
and Section \ref{sec:Quantique} are dedicated to the proof of the
classical and quantum cases, respectively. The main result of theses
two sections is the construction of an appropriate cone such that
the theorems of the previous sections can be applied.

\paragraph*{Acknowledgement:}

I thank my PhD advisor Mathieu Lewin for proposing this problem, useful
discussions and assiduous reading. I am also grateful to Lenaic Chizat
who first mentioned me the Birkhoff-Hopf Theorem. This project has
received funding from the European Research Council (ERC) under the
European Union's Horizon 2020 research and innovation programme (grant
agreement MDFT No 725528).

\section{The Jellium model \label{sec:Classique}}

In this section, we present the Jellium model and state all our results.
The proofs will be given in Section \ref{sub:proofClass} and Section
\ref{sec:Quantique}.

\subsection{The classical Jellium model}

\subsubsection{Mathematical formalism}

We consider $N$ particles of negative charges $q_{1},\text{\dots},q_{N}$
placed on a line $-L<x_{1}<x_{2}<\text{\dots}<x_{N}<L$ in an inhomogeneous
fixed density of charge $\rho\in L^{1}([-L;L])$ such that $\int_{-L}^{L}\rho(s)ds=-\sum_{i=1}^{N}q_{i}$.
The one dimensional solution of $u''=2\delta_{0}$ is $u(x)=-|x|$,
which gives us the total energy of the system 
\begin{align*}
E(x_{1},x_{2},\text{\dots\ ,}x_{N}) & =-\frac{1}{2}\iint_{[-L,L]^{2}}\rho(y_{1})\rho(y_{2})|y_{1}-y_{2}|dy_{1}dy_{2}\\
 & \qquad-\frac{1}{2}\sum_{1\leq i,j\leq N}q_{i}q_{j}|x_{i}-x_{j}|+\sum_{i=1}^{N}q_{i}\int_{-L}^{L}\rho(y)|x_{i}-y|dy.
\end{align*}
The first term is the background-background interaction, the second
term accounts for the electron-electron interaction and the third
term for the background-electron interaction. 

Let us first calculate the state of minimum energy. For each particle
$i$ the position $\tilde{x_{i}}$ which minimizes the energy is such
that
\[
\int_{-L}^{\tilde{x_{i}}}\rho(y)dy=\sum_{1\leq j<i}q_{j}+\frac{q_{i}}{2}.
\]
 It is the condition that for each particle there is the same amount
of charge on its right side and on its left side, such that the particle
is at equilibrium. In the homogeneous case, we have for any $i$,
$\rho|\tilde{x}_{i+1}-\tilde{x}_{i}|=q$. Therefore at $T=0$, the
electrons are located on $\frac{q}{\rho}\mathbb{Z}$ (the Wigner cristal).
But for a general background the lattice is not necessarily a solution.

We subtract the minimum of the energy and rewrite it as

\[
E(x_{1},\text{\dots},x_{N})=E(\tilde{x}_{1},\text{\dots},\tilde{x}_{N})+2\sum_{i=1}^{N}q_{i}\int_{\tilde{x}_{i}}^{x_{i}}\rho(y)(y-x_{i})dy.
\]
We denote by 
\[
U_{i}(s)=-2q_{i}\int_{\tilde{x}_{i}}^{\tilde{x_{i}}+s}(y-\tilde{x}_{i}-s)\rho(y)dy
\]
the potential felt by the $i^{th}-$particle around its stable position. 

We are interested in the canonical model at positive temperature.
The position of the particles $x_{i}$ are now random and the probability
of a set of positions $(x_{i})_{i=1,...,N}$ is proportional to $e^{-\beta E(x_{1},\cdots,x_{N})}$
(Gibbs measure). 

The relevant physical properties of the system are obtained from the
partition function given by 

\[
\mathcal{Z}_{N}(\beta)=e^{-\beta E(\tilde{x}_{1},\text{\dots},\tilde{x}_{N})}\idotsint_{-L<x_{1}<x_{2}<\text{\dots}<x_{N}<L}\prod_{i=1}^{N}e^{-2\beta q_{i}\int_{\tilde{x}_{i}}^{x_{i}}\rho(y)(y-x_{i})dy}dx_{i}
\]
and its free energy per particle
\[
f_{N}(\beta)=-\frac{1}{N\beta}\log(\mathcal{Z}_{N}(\beta)).
\]
We also introduce the marginals $\rho_{I}(x_{i_{1}},x_{i_{2}},\text{\dots},x_{i_{k}})$,
for the probability of the positions of the $k$ particles of the
subset $I=\text{\{}i_{1},i_{2},\cdots,i_{k}\}\subset\text{\{}1,\cdots,N\text{\}}$.
More rigorously, it is the unique function such that for all test
functions $g\in L^{\infty}([-L,L]^{k})$, 

\begin{align*}
 & \iiint g(x_{i_{1}},x_{i_{2}},\cdots,x_{i_{k}})\rho_{k}(x_{i_{1}},\cdots,x_{i_{k}})dx_{i_{1}}\cdots dx_{i_{k}}\\
 & \quad=\frac{e^{-\beta E(\tilde{x}_{1},\text{\dots},\tilde{x}_{N})}}{\mathcal{Z}_{N}(\beta)}\iint_{-L<x_{1}<x_{2}<\text{\dots}<x_{N}<L}g(x_{i_{1}},\cdots,x_{i_{k}})\times\\
 & \quad\qquad\times\prod_{k=1}^{N}e^{-2\beta q_{i}\int_{\tilde{x}_{i}}^{x_{i}}\rho(y)(y-x_{i})dy}dx_{i}.
\end{align*}
Our main interest is to know whether the particles are strongly correlated
or not. This can be quantified by looking at the truncated correlation
functions, which we introduce below.
\begin{defn}
(Cluster property) We say that 
\begin{itemize}
\item the particles are \emph{independent} if 
\[
\rho_{\text{\{}1,\cdots,N\text{\}}}(x_{1},...,x_{N})=\prod_{i=1}^{n}\rho(x_{i}),
\]

\item the particles satisfy a \emph{cluster property} if there exists $I\cup J=\text{\{}1,\cdots,n\text{\}},I\cap J=\emptyset$
such that 
\[
\rho_{\text{\{}1,\cdots,N\text{\}}}(x_{1},...,x_{n})=\rho_{|I|}((x_{i})_{i\in I})\rho_{|J|}((x_{j})_{j\in J}).
\]

\end{itemize}
\end{defn}
In order to characterize the ``clusters'' we introduce the truncated
marginal:
\begin{defn}
(Truncated marginal) The \emph{truncated marginals} $\rho_{k}^{T}$
are defined recursively as follows:
\[
\rho_{J}^{T}(x_{j_{1}},...,x_{j_{k}})=\rho_{J}(x_{j_{1}},...,x_{j_{k}})-\sum_{I_{1}\cup I_{2}\cup...\cup I_{r}=J}\prod_{l=1}^{r}\rho_{I_{l}}^{T}((x_{i})_{i\in I_{l}}).
\]

\end{defn}
The truncated marginals appear to be the good indicator for clustering
properties. Indeed we have the following proposition.
\begin{prop}
\label{PropMarginalTronque}If $\rho_{n}(x_{1},...,x_{n})=\rho_{|I|}((x_{i})_{i\in I})\rho_{|J|}((x_{j})_{j\in J})$
then for all $I'$ such that $I'\cap I\neq\emptyset$ and $I'\cap J\neq\emptyset$
then $\rho_{|I'|}^{T}((x_{i})_{i\in I'})=0$. 
\end{prop}
For the reader's convenience we have written the proof of Proposition
\ref{PropMarginalTronque} in Appendix A. We are now ready to state
our main results.

\subsubsection{Main results}

In the classical case we make the following assumptions:
\begin{itemize}
\item (H1) There exist $q,Q>0$ such that for all $i,0<q\leq q_{i}\leq Q$
. 
\item (H2) There exists $0<m<M$ such that for all $t\in[-L,L],m\leq\rho(t)\leq M$. 
\end{itemize}
These assumptions (H1,H2) imply the following bounds for the potential:
\[
U_{i}(s)\geq qms^{2}
\]
 and 
\[
\frac{d}{ds}U_{i}(s)\geq smq.
\]

Our first result is to be understood as follows: If we consider particles
which are far away from each other (meaning that there are a lot of
other particles between them) then the marginal is exponentially close
to the independent marginal. We also get the cluster property: if
groups of particles are far from each others, then the marginal is
exponentially close to the independent cluster marginal. 
\begin{thm}
\label{Theo-decay-correlation-Jelium-Classic}For any $\beta>0$,
there exists $\kappa<1$ such that for any $I\subset\text{\{}0,\cdots,N\text{\}},$
$|I|=k$, we have

\[
\left|\rho_{I}(x_{i_{1}},x_{i_{2}},\text{\dots},x_{i_{k}})-\prod_{i_{l}\in I}\rho_{\text{\{}i_{l}\text{\}}}(x_{i_{l}})\right|\leq C_{k}\kappa^{d}
\]
for some $C_{k}>0$, provided that between any two consecutive particle
in $I$ there are at least $d$ others particles (in practice take
$d=\inf|i_{l}-i_{l+1}|-1)$. Also we have
\[
|\rho_{I}^{T}(x_{i_{1}},x_{i_{2}},\text{\dots},x_{i_{k}})|\leq C_{k}\kappa^{D}
\]
when there exist two consecutive particles in $I$ with at least $D$
others particles between them (in practice take $D=\max|i_{l}-i_{l+1}|-1)$. 
\end{thm}
Our next result concerns the regularity of the free energy, which
is a fundamental property for one-dimensional systems in statistical
physics. 
\begin{thm}
\label{Theo-Smooth-Jelium-Classical}For any $\beta_{0}>0$, there
exists $\Delta\beta>0$ such that the free energy is smooth on $[\beta_{0}-\Delta\beta,\beta_{0}+\Delta\beta]$
uniformly on $N$. More precisely we have

\[
|\frac{d^{k}}{d\beta^{k}}f_{N}|\leq M_{k},
\]
with $M_{k}>0$ independent of $N$ and for all $\beta\in[\beta_{0}-\Delta\beta,\beta_{0}+\Delta\beta]$.
\end{thm}
In the proof we will show the following estimate on $M_{k}$:
\begin{equation}
M_{k}\leq k!c^{k}\label{eq:derivateEstime}
\end{equation}
for some $c>0$. From this bound we obtain the analyticity of the
limiting free energy, when this limit exists.
\begin{thm}
\label{Theo-Smooth-Jelium-Classical-1}For any $\beta_{0}>0$, there
exists $\Delta\beta>0$ such that if there exists $f$ such that a
$f_{N_{k}}\rightarrow f$ on $[\beta_{0}-\Delta\beta,\beta_{0}+\Delta\beta]$
for a subsequence $N_{k}\rightarrow\infty$, then $f$ is real analytic
on $[\beta_{0}-\Delta\beta,\beta_{0}+\Delta\beta]$.
\end{thm}
As a corollary, the system will not reveal any phase transition for
$\beta\neq\infty$. 
\begin{cor}
If the charge background is periodic or if it is constructed randomly
with an ergodic process, there exists a limiting function $f$ such
that $f_{N}\rightarrow f$ (almost surely in the ergodic case) and
$f$ is real analytic on $(0,\infty)$.
\end{cor}
This generalizes the results of Kunz \cite{KUNZ1974303}.

\subsection{The quantum model}

\subsubsection{Mathematical formalism}

We now give our results for the quantum problem. In the classical
case, we neglect the kinetic energy because in phase space momentum
and position are independent for the Gibbs measure. This is no longer
true in the quantum case and we have to consider the whole $N$-particle
Hamiltonian 

\[
H=-\frac{1}{2}\sum_{i=1}^{N}\partial_{x_{i}}^{2}+E(x_{1},\cdots,x_{N}).
\]
For simplicity we choose Dirichlet boundary conditions at the two
ends $\mp L$. The quantum fermionic canonical function is 

\[
\mathcal{Z}_{N}^{Q}(\beta)=\text{Tr}(\exp(-\beta H))
\]
and the free energy is 
\[
f_{N}^{Q}(\beta)=-\frac{1}{\beta N}\log(\mathcal{Z}_{N}^{Q}(\beta)).
\]

We have the following Feynman-Kac formula \cite{jansen2014wigner}
for the partition function $\mathcal{Z}_{N}^{Q}(\beta)$.
\begin{prop}
(Feynman Kac formula)\label{(Feynman-Kac-formula)} We have 

\begin{equation}
\mathcal{Z}_{N}^{Q}(\text{\ensuremath{\beta}})=\int_{\mathcal{X}}\mu_{x_{1}x_{\text{1}}}\text{\dots}\mu_{x_{N}x_{N}}(e^{-\int_{0}^{\beta}U(\gamma_{1}(t),\text{\dots},\gamma_{N}(t))dt}1_{(\gamma_{1},\cdots,\gamma_{N})\in W_{N}})dx_{1}\text{\dots}dx_{N}\label{eq:FeynmanKac}
\end{equation}
 and
\[
\rho(\boldsymbol{x};\boldsymbol{y})=\frac{1}{Z_{N}}\mu_{x_{1}y_{1}}\times\mu_{x_{2},y_{2}}\times\text{\dots}\times\mu_{x_{N},y_{N}}(e^{-\int_{0}^{\beta}U(\gamma_{1}(t),\text{\dots},\gamma_{N}(t))dt})
\]
where $\mathcal{X}=\text{\{}(x_{1},\cdots,x_{N}):-L<x_{1}<\cdots<x_{N}<L\text{\}}$, 

\[
W_{N}=\text{\{}(\gamma_{1},\text{\dots},\gamma_{N})|\forall t\in[0,\beta]:-L<\gamma_{1}(t)<\gamma_{2}(t)<\text{\dots}<\gamma_{N}(t)<L\text{\}}
\]
is the Weyl chamber and $\mu_{x,y}$ are the probability measures
of a Brownian bridge from $x$ to $y$ of length $\beta$.
\end{prop}
The random system we study in the quantum model is no longer the positions
$(x_{i})_{i\leq N}$ but rather the paths $(\gamma_{i})_{i\leq N}$.
We define the extended marginals on the set of paths $\rho^{\Gamma}(\gamma_{1},...,\gamma_{N})$
and we are able to apply the theorems of Section \ref{sec:Decay-of-correlation}
in this set up. However, for simplicity we will only states the results
on the position marginals $\rho_{k}(x_{i_{1}},...,x_{i_{k}})$ which
satisfy, for any bounded function $g:[-L,L]^{N}\rightarrow\mathbb{R}$,
\begin{align*}
 & \idotsint\rho_{k}(x_{i_{1}},\cdots,x_{i_{k}})g(x_{i_{1}},\cdots,x_{i_{k}})dx_{1}\cdots dx_{k}\\
 & \quad=\frac{1}{Z_{N}^{Q}(\beta)}\int_{-L<x_{1}<\cdots<x_{N}<L}\mu_{x_{1}x_{\text{1}}}\text{\dots}\mu_{x_{N}x_{N}}(e^{-\int_{0}^{\beta}U(\gamma_{1}(t),\text{\dots},\gamma_{N}(t))dt}1_{(\gamma_{1},\cdots,\gamma_{N})\in W_{N}})\\
 & \quad\quad g(x_{i_{1}},\cdots,x_{i_{k}})dx_{1}\text{\dots}dx_{N}
\end{align*}

\subsubsection{Main results}

Unfortunately, in the quantum case we are only able to prove a result
in a perturbation regime where $\rho$ and the $q_{i}$ are almost
constant. We therefore make the following assumptions :
\begin{itemize}
\item (HQ1) $\text{ }q(1-\epsilon)\leq q_{i}\leq q(1+\epsilon)$ for all
$i$.
\item (HQ2) $\rho(1-\epsilon)\leq\rho(t)\leq\rho(1+\epsilon)$ for all $t$. \end{itemize}
\begin{thm}
\label{Theo-decay-correlation-quantum}For any $\beta>0$, under condition
(HQ1-2) for $\epsilon>0$ small enough, there exists $\kappa<1$,
such that for all $I\subset\text{\{}1,\cdots,N\text{\}}$, $|I|=k$, 

\[
|\rho_{I}(x_{i_{1}},x_{i_{2}},\text{\dots},x_{i_{k}})-\prod_{i_{l}\in I}\rho_{\text{\{}i_{l}\text{\}}}(x_{i_{l}})|\leq C_{k}\kappa^{d}
\]
for some $C_{k}>0$, if between any two consecutive particle in $I$
there are at least $d$ others particles (in practice take $d=\inf|i_{l}-i_{l+1}|-1)$.On
the other hand, 
\[
|\rho_{I}^{T}(x_{i_{1}},x_{i_{2}},\text{\dots},x_{i_{k}})|\leq C_{k}\kappa^{D}
\]
if there exists two consecutive particles in $I$ with at least $D$
others particles between them (in practice take $D=\max|i_{l}-i_{l+1}|-1)$. 
\end{thm}
As in the classical case we obtain the regularity of the partition
function $f_{N}^{Q}$.
\begin{thm}
\label{Theo-smooth-quantum}For any $\beta_{0}>0$, there exists $\Delta\beta>0$
such that under condition (HQ1-2) with $\epsilon>0$ small enough,
the free energy is $C^{\infty}$ on $[\beta_{0}-\Delta\beta,\beta_{0}+\Delta\beta]$
and for all $k$ we have
\[
\big|\frac{d^{k}}{d\beta^{k}}f_{N}^{Q}\big|\leq M_{k}
\]
with $M_{k}$ independent of $N$.
\end{thm}
Finally we can prove analyticity of the free energy with the same
estimate as (\ref{eq:derivateEstime}).
\begin{thm}
\label{Theo-smooth-quantum-1}For any $\beta_{0}>0$, there exists
$\Delta\beta>0$ such that under condition (HQ1-2), for $\epsilon>0$
small enough, if $f_{N_{k}}$ admits a limit $f$ for a subsequence
$N_{k}\rightarrow\infty$, then $f$ is real analytic on $[\beta_{0}-\Delta\beta,\beta_{0}+\Delta\beta]$.\end{thm}
\begin{cor}
We make the same assumptions as in Theorem \ref{Theo-decay-correlation-quantum}.
If the charge background is periodic or if it is constructed randomly
with an ergodic process, there exists a limiting function $f$ such
that $f_{N}\rightarrow f$ (almost surely in the ergodic case) and
$f$ is real analytic.
\end{cor}

\section{General theory to apply the Birkhoff-Hopf theorem}

\label{sec:BHT}The Birkhoff-Hopf theorem has been used for instance
to study non linear integrable equations, weak ergodic theorems, or
the so-called $DAD$ problem \cite{borwein1994entropy}.

We first introduce the notion of cone and the Hilbert distance. In
this set up we can state the Birkhoff-Hopf theorem. Then we prove
Theorem \ref{Theo-decay-correlation-Jelium-Classic} and Theorem \ref{Theo-Smooth-Jelium-Classical}
with the extra assumption of strictly contracting operators.

\subsection{Framework and Birkhoff-Hopf theorem. \label{sub:Framework-and-Birhoff}}

We follow \cite{eveson1995elementary} for the notation and we refer
to this paper for a proof of the Birkhoff Hopf theorem (Theorem \ref{thm:(Birkhoff-Hopf-)-}
bellow). Let $E$ be a real linear Banach space.
\begin{defn}
(Abstract cone) $\mathcal{C}\subset E$ is called a \emph{cone} if
\begin{enumerate}
\item $\mathcal{C}$ is convex,
\item $\lambda\mathcal{C}\subset\mathcal{C}$ for any $\lambda\geq0$,
\item $\mathcal{C}\cap-\mathcal{C}=\text{\{}0\text{\} }$.
\end{enumerate}
\end{defn}
Using $\mathcal{C}$ we define a partial order on $E$
\begin{defn}
(Partial order) For any $x,y\in E$, we write $x\leq_{\mathcal{C}}y$
if $y-x\in\mathcal{C}$.
\end{defn}
For clarity we will use $\leq$ instead of $\leq_{\mathcal{C}}$ if
there is no confusion about the cone.
\begin{defn}
If $\mathcal{C}$ is a cone, we define the dual cone $\mathcal{C}^{*}$
by

\[
\mathcal{C}^{*}:=\text{\{}f\in E^{*}:\forall x\in\mathcal{C},(f,x)\geq0\text{\}}.
\]

\end{defn}
The set $\mathcal{C}^{*}$ is a cone if $\mathcal{C}-\mathcal{C}$
is dense and in particular if $\mathcal{C}$ has nonempty interior.
We say that $x,y\in\mathcal{C}$ are \emph{comparable and write} $x\text{\ensuremath{\sim}}y$
if there exist $\alpha,\beta>0$ such that $\alpha x\leq y\leq\beta x$.
This defines an equivalence relation. We say that $\mathcal{C}$ is
\emph{normal }if there exists $\gamma>0$ such that 
\[
\forall x,y\in\mathcal{C},0\leq x\leq y\Rightarrow\|x\|\leq\gamma\|y\|.
\]

\begin{defn}
For any $x,y\in\mathcal{C}$ comparable, we define the Hilbert metric
by

\[
d_{\mathcal{C}}(x,y)=\log\frac{\beta_{\min}(x,y)}{\alpha_{\max}(x,y)}
\]
where

\[
\alpha_{\max}(x,y)=\sup\text{\{}\alpha>0:\alpha x\leq y\text{\}}
\]
 and

\[
\beta_{\min}(x,y)=\inf\text{\{}\beta>0:y\leq\beta x\text{\}}
\]

\end{defn}
The Hilbert metric is a metric on the projective space of $\mathcal{C}$.

We say that $T:E\rightarrow E$ is \emph{order-preserving }if $x\leq y\Rightarrow T(x)\leq T(y)$.
If $T$ is a linear operator (the only case we will consider here)
this is equivalent to $T(\mathcal{C})\subset\mathcal{C}$. 
\begin{rem}
If $T$ is order-preserving then $T$ is non-expanding for the Hilbert
metric. Indeed $\alpha x\leq y\leq\beta x$ implies $\alpha T(x)\leq T(y)\leq\beta T(x)$.

We introduce the projective diameter

\[
\Delta(T)=\sup\text{\{}d_{\mathcal{C}}(T(x),T(y)):x,y\in\mathcal{C},\quad T(x)\sim T(y)\text{\}}
\]
 and the contracting ratio
\[
\kappa(T)=\inf\text{\{}c>0:\forall x,y\text{ }d_{\mathcal{C}}(T(x),T(y))\leq cd_{\mathcal{C}}(x,y),\quad T(x)\sim T(y)\text{\}}
\]

\end{rem}
Here is the main theorem we will use : 
\begin{thm}
(Birkhoff-Hopf \cite{birkhoff1957extensions,hopf1963inequality})
If $T$ is order-preserving then \label{thm:(Birkhoff-Hopf-)-}

\[
\kappa(T)=\tanh\left(\frac{\Delta(T)}{4}\right).
\]

\end{thm}
The result has to be understood as follows: if the image of the cone
of the order preserving operator is strictly inside the cone ($\Delta(T)<\infty$),
then the operator is strictly contracting $(\kappa(T)<1)$ for the
Hilbert metric.

\subsection{Application to statistical physics}

We now use the previous formalism to study the partition function
and the marginals from statistical physics in the abstract framework
of positive operators.
\begin{defn}
(Density function) Let $u\in\mathcal{C}^{*}$ and $v\in\mathcal{C}$,
and let $X,Y,(X_{i})_{i}$ be positive bounded operators. We define
\begin{itemize}
\item the partition function by 
\[
\mathcal{Z}=(u,T_{N}\text{\dots}T_{0}v),
\]

\item the one-point density function by 
\[
\rho_{K_{1}}(X)=\frac{1}{\text{\ensuremath{\mathcal{Z}}}}(u,T_{N}\text{\dots}T_{K_{1}+1}XT_{K_{1}}\text{\dots}T_{0}v),
\]

\item the pair correlation function by 
\[
\rho_{K_{2},K_{1}}(Y,X)=\frac{1}{\mathcal{Z}}(u,T_{N}\text{\dots}T_{K_{2}+1}YT_{K_{2}}\text{\dots}T_{K_{1}+1}XT_{K_{1}}\text{\dots}T_{0}v),
\]

\item the $k$-point correlation function by
\[
\rho_{K_{k},...,K_{2},K_{1}}(X_{k},...,X_{1})=\frac{1}{\text{\ensuremath{\mathcal{Z}}}}(u,T_{N}\text{\dots}T_{K_{k}+1}X_{k}T_{K_{k}}\text{\dots}T_{K_{1}+1}X_{1}T_{K_{1}}\text{\dots}T_{0}v).
\]

\end{itemize}
\end{defn}
The operators $X,Y,(X_{i})$ should be thought of as test functions
acting on the position of the $K_{i}^{th}$ particle. 
\begin{rem}
The simplest model that can be written in this formalism is the one-dimensional
Ising model \cite{pfeuty1970one}. All the results stated above for
Jellium can be easily adapted to the inhomogeneous one-dimensional
Ising model.

We also think of Markov processes on a finite or compact set, in which
case $T_{i}$ is the transitive kernel from $X_{i}$ to $X_{i+1}$.
\end{rem}

\subsubsection{Decay of correlations}

The following theorem states the exponential decay of the correlation
functions.
\begin{thm}
(Decay of correlations) Let $(T_{i})_{i=1,\text{\dots},N}$ be positive
operators such that 
\[
\Delta(T_{i}(C))\le M<\infty
\]
 for any $i$. Then there exist $c>0$ which depends only on $k$,
such that for $\min|K_{j+1}-K_{j}|$ large enough, we have

\label{theo-decay-of-correlation}
\begin{align*}
 & \big(1-c\kappa^{\min_{j}|K_{j+1}-K_{j}|}\big)\prod_{i=1}^{k}\rho_{K_{i}}(X_{i})\\
 & \qquad\leq\rho_{K_{k},...,K_{1}}(X_{k},...,X_{1})\leq\big(1+c\kappa^{\min_{j}|K_{j+1}-K_{j}|}\big)\prod_{i=1}^{k}\rho_{K_{i}}(X_{i})
\end{align*}
 with $\kappa=\tanh(\frac{M}{4})$.
\end{thm}
The decay of correlations is an important concept in statistical physics
and it is ubiquitous in one-dimensional systems \cite{ruelle1999statistical}.

The proof of Theorem \ref{theo-decay-of-correlation} is provided
below in Section \ref{sec:Decay-of-correlation}.

\subsubsection{Regularity of the free energy}

The second theorem states that the partition function depending on
a parameter is smooth, if the transitive operators are smooth enough.
In order to express the ``regularity'' of the operator in the framework
of a cone and the Hilbert distance, we have to construct the following
norm. The following result says that the distance is close to being
a norm in the neighborhood of any point $x_{0}$. 
\begin{prop}
\label{prop:NormEqui-1}Let $x_{0}\in\mathcal{C}$. For any $\epsilon>0$,
there exists $r>0$, a function $f$ and a norm $\mathcal{N}$ defined
on the projective space, such that $d$ can be written as follows 

\[
d(x,y)=f(x,y)\mbox{\ensuremath{\mathcal{N}}}(y-x)
\]
for all $x,y$ such that $d(x,x_{0})<r$ and $d(y,x_{0})<r,$ with
$|f(x,y)-1|<\epsilon$.
\end{prop}
We can now state our second main result
\begin{thm}
\label{thm:Smooth-1-1}Let $\mathcal{C}$ be a cone and let $T_{i}(\beta)$
be a family of smooth bounded operators for $\beta$ in the neighborhood
$[\beta_{0}-\delta,\beta_{0}+\delta]$ of $\beta_{0}$, which are
contracting of parameter $\kappa<1$, uniformly in $\beta$ and $i$.
For all $i$, we denote by $\mathcal{N}_{i}$ the norms defined in
Proposition \ref{prop:NormEqui-1} around $x_{i}=\prod_{j=0}^{i-1}T_{j}(\beta)b$.
Assume that the derivatives in $\beta$ of the operator are uniformly
bounded for theses norm, that is, 
\[
\exists C',\forall i,\quad\left\Vert \frac{d^{k}T_{i}}{d\beta^{k}}\right\Vert _{\mathcal{N}_{i}\rightarrow\mathcal{N}_{i+1}}\leq C_{k}'
\]
for some constant $C_{k}'$ independent of $i$ and of $\beta\in[\beta_{0}-\delta,\beta_{0}+\delta]$.
Then 

\[
f_{N}(\beta)=\frac{1}{N}\log\left\langle a,\prod_{i=0}^{N-1}T_{i}(\beta)b\right\rangle 
\]
 is uniformly smooth, meaning there is a constant $C$ which depends
only on $\kappa$ and $(M_{k})$ such that : 

\[
\Big|\frac{d^{n}f_{N}(\beta)}{d\beta^{n}}\Big|\leq C(\kappa,(M_{k})_{k\leq n})
\]
 where $M_{k}=\sup_{i,\beta\in[\beta_{0}-\delta,\beta_{0}-\delta]}\|\frac{d^{k}T_{i}}{d\beta^{k}}\|_{\mathcal{N}_{i}\rightarrow\mathcal{N}_{i+1}}$
. 

Moreover if the following limit exists 

\[
f(\beta)=\lim_{k\rightarrow\infty}\frac{1}{N_{k}}\log\left\langle a,\prod_{i=0}^{N_{k}-1}T_{i}(\beta)b\right\rangle 
\]
 for a sequence $N_{k}\rightarrow\infty$, then it is smooth in a
neighborhood of $\beta_{0}$: 

\[
\left|\frac{d^{n}f(\beta)}{d\beta^{n}}\right|\leq C(\kappa,(M_{k})_{k\leq n}).
\]

\end{thm}
If the positive operator appears to be uniformly analytic for the
constructed norm then the free energy is analytic. More precisely
we have the following theorem
\begin{thm}
With the same assumptions as in Theorem \ref{thm:Smooth-1-1}\textup{,}
\textup{if there exists $r\geq0$ such that} 
\[
\frac{\|\partial_{\beta}^{n}T_{i}\|_{\mathcal{N}_{i}\rightarrow\mathcal{N}_{i+1}}}{n\text{!}}\leq r^{n}
\]
 for all $n$, then $f$ is real analytic around $\beta$ with radius
of convergence at least equal to $(1-\kappa)/r$\label{TheoremAnalytic}.
\end{thm}
The two theorems of this section are proved later in Section \ref{sec:Smoothness-of-the}.

\subsubsection{Central Limit Theorem}

We consider the particular case where the space is $L^{1}(\Lambda)$,
with $\Lambda$ a measurable set and the cone is $\mathcal{C}=\text{\{}f\in L^{1}(\Lambda):f\geq0\text{\}}$.
We construct the canonical random process $y_{i}$ as follows. Let
$A_{1},\cdots,A_{N}\subset\Lambda$, and take as test functions $1_{A_{1}},1_{A_{2}}...,1_{A_{N}}$.
Then we define
\begin{equation}
\mathbb{P}(y_{1}\in A_{1},y_{2}\in A_{2},...,y_{N}\in A_{N})=\rho_{1,\cdots,N}(1_{A_{1}},...,1_{A_{N}})=\frac{1}{Z}\left\langle \prod_{i=1}^{N}T_{i}1_{A_{i}}\right\rangle .\label{eq:ProbaCone}
\end{equation}

The decay of correlation in Theorem \ref{theo-decay-of-correlation}
is the mixing property of the process $(y_{i})$.

The Central Limit Theorem has been proved for a urge number of random
processes like martingales \cite{hall2014martingale}, Markov processes\cite{herve2008vitesse,guivarc1988theoremes}
or random products of matrices \cite{le1982theoremes}. One of the
classical proofs of the central limit theorem uses the regularity
of the Laplace transform, this is what we adapt here.
\begin{thm}
(Central Limit)\label{Theorem(TCL)-1} Let $h_{i}:\Lambda\rightarrow\mathbb{R}$
be such that $\mathbb{E}(\exp(h_{i}(y_{i})))<\infty$ where the mean
is on the probability (\ref{eq:ProbaCone}). Let $T_{i}(\beta)=e^{\beta h_{i}(y_{i})}T_{i}$.
If the $T_{i}(\beta)$ satisfy the assumptions of Theorem \ref{thm:Smooth-1-1},
then we have
\[
\sup_{x\in\mathbb{R}}\left|\mathbb{P}\left(\frac{1}{\sqrt{N}}\sum(h_{i}(y_{i})-\mathbb{E}(h_{i}(y_{i})))\geq x\right)-\mathbb{P}\big(\mathcal{N}(0,\sigma^{2}\big)\geq x)\right|=O\left(\frac{1}{\sqrt{N}}\right)
\]
where $\sigma^{2}$ is the second derivative of the free energy 
\[
\sigma^{2}=\frac{d^{2}}{d\beta^{2}}\left[\frac{1}{N}\log\left\langle a,\prod_{i=0}^{N-1}T_{i}(\beta)b\right\rangle \right].
\]

\end{thm}
This theorem is proved in Section \ref{sect.TCL}.

\subsection{Rank-one operator approximation\label{sec:Rank-one-operator}}

One of the first historical applications of the Birkhoff-Hopf Theorem
was in population demography \cite{cohen1979ergodic}. An age structure
diagram $f$ evolves due to birth and death, with death and birth
rates not constant in time and one can calculate its time evolution.
It appears that even if $f$ does not converge to an equilibrium,
the long time evolution is independent of the initial age structure
$f(0)$. Namely this is a weak ergodicity property: if $f_{1}$ and
$f_{2}$ are two solutions of the evolution with different initial
data, $\|f_{1}(t)/\|f_{1}(t)\|-f_{2}(t)/\|f_{2}(t)\|\|\rightarrow0$. 

In this section, we formulate weak ergodicity in term of a rank-one
operator approximation and we give a construction and an estimate
of such an approximation in case where several contracting operators
are composed one after another.

\subsubsection{The cone of order preserving operators}

We state here some simple results about the set of order preserving
operators. 
\begin{lem}
Let $\mathcal{C}$ with $\mathcal{C}-\mathcal{C}$ dense. The set
of corresponding order-preserving operator is a cone. 
\end{lem}
We denote by $\mathcal{P}_{\mathcal{C}}$ this cone and only $\mathcal{P}$
if there is no confusion.
\begin{proof}
We check every point of the definition.
\begin{enumerate}
\item If $A,B$ are order preserving operator then $A+B$ is an order preserving
operator. Indeed $(A+B)(x)\in\mathcal{C}$ for all $x\in\mathcal{C}$.
\item The set of order preserving operator is invariant by product of strictly
positive scalars.
\item Let $A\in\mathcal{P}\cap-\mathcal{P}$, then $(f,Ax)=0$ for all $x\in\mathcal{C}$
and all $f\in\mathcal{C}^{*}$. Therefore $(f_{1}-f_{2},A(x_{1}-x_{2}))=0$
for all $x_{1},x_{2}\in\mathcal{C}$ and all $f_{1},f_{2}\in\mathcal{C}^{*}$.
Therefore $A=0$ since $\mathcal{C}^{*}-\mathcal{C}^{*}$ and $\mathcal{C}-\mathcal{C}$
are dense.
\end{enumerate}
\end{proof}
\begin{example}
One can think of $\mathcal{C}$ the positive vectors in $\mathbb{R}^{n}$
and the set of matrices $\mathcal{M}_{n}(\mathbb{R})$ with positive
coefficients.
\end{example}
We have the following order on the set of operators : 
\[
B\geq_{\mathcal{P}}A\Leftrightarrow(B-A)(\mathcal{C})\subset\mathcal{C}
\]
 and the corresponding Hilbert distance 
\[
d_{\mathcal{P}}(A,B)=\min\Big(\log\big(\frac{\beta}{\alpha}\big):\alpha A\leq B\leq\beta A\Big).
\]
 
\begin{rem}
If $A\leq_{\mathcal{P}}B$ and $C\leq_{\mathcal{P}}D$ then $AC\leq_{\mathcal{P}}BD$.
Indeed $(B-A)C(\mathcal{C})\subset(B-A)(\mathcal{C})\subset\mathcal{C}$,
and $B(D-C)(\mathcal{C})\subset B(\mathcal{C})\subset\mathcal{C}$.
Therefore $BD\geq_{\mathcal{P}}BC\geq_{\mathcal{P}}AC$. 
\end{rem}
Unfortunately that $T$ is contracting does not imply that $\tilde{T}:A\rightarrow TA$
is contracting as well. One can take for example: 
\[
A=\begin{pmatrix}1 & 2\\
2 & 1
\end{pmatrix}\text{,}\quad B=\begin{pmatrix}2 & 1\\
1 & 2
\end{pmatrix}\text{ and }T=\begin{pmatrix}1 & 0\\
0 & 0
\end{pmatrix},
\]
in which case $d_{\mathcal{P}}(A,B)=\log(\frac{2}{(1/2)})=\log(4)$
and $d_{\mathcal{P}}(TA,TB)=\log(4)$ as well.
\begin{lem}
Let $A,B,C,D$ be increasing operators. Then 
\[
d_{\mathcal{P}}(AB,CD)\leq d_{\mathcal{P}}(A,C)+d_{\mathcal{P}}(B,D).
\]
\label{prop:distanceProduit}\end{lem}
\begin{proof}
Let $\alpha_{1},\beta_{1},\alpha_{2},\beta_{2}$ such that $\alpha_{1}A\leq C\leq\beta_{1}A$,
$\alpha_{2}B\leq D\leq\beta_{2}B$ and $\log\big(\frac{\beta_{1}}{\alpha_{1}}\big)-d_{\mathcal{P}}(A,C)\leq\epsilon$,
$\log\big(\frac{\beta_{2}}{\alpha_{2}}\big)-d_{\mathcal{P}}(B,D)\leq\epsilon$.
Then $\alpha_{1}\alpha_{2}AC\leq BD\leq\beta_{1}\beta_{2}AC$ and
we have $d_{\mathcal{P}}(AB,CD)\leq\log\big(\frac{\beta_{1}\beta_{2}}{\alpha_{1}\alpha_{2}}\big)\leq d_{\mathcal{P}}(A,C)+d_{\mathcal{P}}(B,D)+2\epsilon$.
\end{proof}
Now we construct a rank-one operator $L=z\cdot l$, with a vector
$z\in E$, and a linear form $l\in E^{*}$ to approximate a contracting
function $T$. It is natural to choose $z\in T(C)$. We construct
$l$ in the following subsection.

\subsubsection{Rank-one operator construction}

We construct here the rank-one operator close to a contracting operator.
\begin{lem}
Let $a:C\rightarrow\mathbb{R}_{+}$ and $b:C\rightarrow\mathbb{R}_{+}$
be such that \label{lem:ConsLinear}
\begin{enumerate}
\item there exist $M_{1},M_{2}<\infty$, for all $x$, $a(x)\leq M_{1}\|x\|$
and $b(x)\leq M_{2}\|x\|$ ,
\item for all $\lambda\geq0\text{ and }x\in C$, $a(\lambda x)=\lambda a(x)$
and $b(\lambda x)=\lambda b(x)$,
\item for all $x\in C$ $a(x)\leq b(x)$,
\item for all $x,y\in C$ $a(x+y)\geq a(x)+a(y)$ and $b(x+y)\leq b(x)+b(y).$
\end{enumerate}
Then there exists a linear form $l\in C^{*}$ such that, for any $x\in C$,
\[
a(x)\leq l(x)\leq b(x).
\]
\end{lem}
\begin{proof}
We check that $b$ is a convex function, 

\[
b(tx+(1-t)y)\leq b(tx)+b((1-t)y)=tb(x)+(1-t)b(y),
\]
and that $a$ is a concave function,

\[
a(tx+(1-t)y)\geq a(tx)+a((1-t)y)=ta(x)+(1-t)a(y).
\]

Let us define two sets : $B=\text{\{}(x,s)\in C\times\mathbb{R}:b(x)\leq s\text{\}}$
and $A=\text{\{}(x,s)\in C\times\mathbb{R}:a(x)>s\text{\}}$. Then
$A\cap B=\emptyset$, $A$ and $B$ are convex. Because of the Hahn-Banach
separation theorem, there exists $l\neq0$ a linear form on $E\times\mathbb{R}$
such that for all $(x,s)\in B$, $l(x,s)\geq0$ and all $(x,s)\in A$,
$l(x,s)\leq0$. We have $l(x,s)=l_{1}(x)+\alpha s$ with $l_{1}\in E^{*}$
and $\alpha\in\mathbb{R}$. 

We then prove that $\alpha>0$. For any $s_{0}>0$ we have $(0,s_{0})\in B$,
$(0,-s_{0})\in A$, $\alpha s_{0}\geq-\alpha s_{0}$ and as a conclusion
$\alpha\geq0$. If $\alpha=0$, because for any $x\in C$ there exist
$s_{1}$ and $s_{2}$ such that $(x,s_{1})\in B$ and $(x,s_{2})\in A$,
we have $0\geq l(x,s_{2})=l_{1}(x)=l(x,s_{1})\geq0$ and therefore
$l_{1}(x)=0$. Let $x_{0}\in\mathring{C}$, and $V_{\epsilon}(x_{0})\subset C$
a small ball with center $x_{0}$ and radius $\epsilon$. Then for
all $y$ with $\|y\|<\epsilon,$ we have $l_{1}(y)=l_{1}(y+x_{0})=0$.
As a conclusion $l=l_{1}=0$ which is absurd, so $\alpha\neq0$. 

Let $l_{0}=-\frac{l_{1}}{\alpha}$. Since $(x,b(x))\in B$, $-l_{0}(x)+b(x)\geq0$
we have $l_{0}(x)\leq b(x)$. Moreover for $\epsilon>0$, $(x,a(x)-\epsilon)\in A$
$-l_{0}(x)+a(x)-\epsilon\leq0$ and therefore $l_{0}(x)\geq a(x)$
. \end{proof}
\begin{cor}
There exists a rank one operator $L_{T}=z\cdot l$ with $z\in\mathcal{C}$
and $l\in\mathcal{C}^{*}$ such that $d_{\mathcal{P}}(T,L_{T})\leq2\Delta(T)$.\end{cor}
\begin{proof}
Let $z=T(y_{0})\in T(\mathcal{C})$ and define $a$ and $b$ as follows:
for any $x\in\mathcal{C}$ 
\[
a(x)=_{def}\max\text{\{}\alpha:\alpha T(y_{0})\leq_{\mathcal{C}}T(x)\mbox{\}}
\]
 and 
\[
b(x)=_{def}\min\text{\{}\beta:T(x)\leq_{\mathcal{C}}\beta T(y_{0})\text{\}.}
\]
 It is possible to check the hypothesis of Lemma \ref{lem:ConsLinear}.
Indeed we have that 
\[
\mbox{\ensuremath{\begin{cases}
 \alpha_{1}T(y_{0})\leq_{\mathcal{C}}T(x_{1})\leq_{\mathcal{C}}\beta_{1}T(y_{0}),\\
 \alpha_{2}T(y_{0})\leq_{\mathcal{C}}T(x_{2})\leq_{\mathcal{C}}\beta_{2}T(y_{0}) 
\end{cases}}}
\]
implies
\[
(\alpha_{1}+\alpha_{2})T(y_{0})\leq_{\mathcal{C}}T(x_{1})+T(x_{2})\leq_{\mathcal{C}}(\beta_{1}+\beta_{2})T(y_{0})
\]
and therefore $a(x_{1}+x_{2})\geq a(x_{1})+a(x_{2})$ and $b(x_{1}+x_{2})\leq b(x_{1})+b(x_{2})$.
We also have 
\[
a(\lambda x)=\lambda a(x)\text{ and }b(\lambda x)=\lambda b(x)
\]
for all $\lambda\geq0\text{ and }x\in C$. We can then apply Lemma
\ref{lem:ConsLinear}: there exists a linear form $l$ with $a(x)\leq l(x)\leq b(x)$.
Moreover $\log\frac{b(x)}{a(x)}\leq\Delta(T)$ for all $x\in\mathcal{C}$.
We then have for all $x$ $\frac{l(x)}{a(x)}\leq e^{-\Delta(T)}$
and $\frac{b(x)}{l(x)}\leq e^{\Delta(T)}$ and therefore 

\[
e^{-\Delta(T)}T(x)\leq_{\mathcal{C}}T(y_{0})\cdot l(x)\leq_{\mathcal{C}}e^{\Delta(T)}T(x).
\]
As a conclusion $d_{\mathcal{P}}(T(y_{0})\cdot l,T)\leq\log(e^{2\Delta(T)})\leq2\Delta(T)$.\end{proof}
\begin{cor}
Let $(T_{i})_{i=0,\text{\dots},N}$ be positive operators. If
\begin{enumerate}
\item $\Delta(T_{0}(\mathcal{C}))\le R<\infty,$
\item $T_{i}$ $i=1,\text{\dots},N$ are uniformly contracting of parameter
$\kappa<1$, 
\end{enumerate}
then there exists a linear form $l$, $z_{0}\in\mathcal{C},$ $\|z_{0}\|=1$
and $l\in\mathcal{C}^{*}$ such that 

\[
d_{\mathcal{P}}\big((T_{N}\text{\dots}T_{0}),z_{0}\cdot l)\leq2\kappa^{N}R.
\]
\end{cor}
\begin{proof}
We have $\Delta(T_{N}\text{\dots}T_{0})\leq\kappa^{N}R$ and the result
follows from the previous corollary.
\end{proof}

\subsection{Decay of correlation function \label{sec:Decay-of-correlation}}

Here we prove Theorem \ref{theo-decay-of-correlation}. The idea is
to replace the product of contracting operator between the $k$ points
of measure by a rank-one operator. We will do so for $k=2$ and for
$k>2$ this will be exactly the same. More precisely we will prove
the following
\begin{thm}
(Theorem \ref{theo-decay-of-correlation} in the case $k=2$)\label{thm:TwoPointCorelation}
Let $(T_{i})_{i=1,\text{\dots},N}$ be positive operators such that
$\Delta(T_{i}(C))\le M<\infty$ for any $i$, and let $K_{1},K_{2}\in\mathbb{N}$
be such that $1\leq K_{1}\le K_{2}\leq N$. Let $u,v\in C$, and $X,Y$
be two positive operators.Then:
\begin{align*}
 & e^{-8R(\kappa^{K_{1}}+\kappa^{K_{2}-K_{1}}+\kappa^{N-K_{2}})}\rho_{K_{1}}(X)\rho_{K_{2}}(Y)\\
 & \qquad\leq\rho_{K_{2},K_{1}}(Y,X)\leq e^{8R(\kappa^{K_{1}}+\kappa^{K_{2}-K_{1}}+\kappa^{N-K_{2}})}\rho_{K_{1}}(X)\rho_{K_{2}}(Y).
\end{align*}

\end{thm}
One can use this theorem for $K_{1},$$K_{2}-K_{1},N-K_{2}$ large.
In this case, the Taylor expansion of $e^{x}$ gives 
\[
|\rho_{K_{1}}(X)\rho_{K_{2}}(Y)-\rho_{K_{1},K_{2}}(X,Y)|\leq16R(\kappa^{K_{1}}+\kappa^{K_{2}-K_{1}}+\kappa^{N-K_{2}})\|X\|\|Y\|,
\]
which decays exponentially.
\begin{proof}
Let us introduce $L_{K_{1}0}=z_{K_{1}0}l_{K_{1}0}$, $L_{K_{2}K_{1}}=z_{K_{2}K_{1}}l_{K_{2}K_{1}}$
and $L_{NK_{2}}=z_{NK_{2}}l_{NK_{2}}$ which are rank-one operators
such that 
\[
\begin{cases}
d_{\mathcal{P}}\big((T_{K_{1}}\text{\dots}T_{0}),L_{K_{1}0})\leq2\kappa^{K_{1}}R,\\
d_{\mathcal{P}}\big((T_{K_{2}}\text{\dots}T_{K_{1}}),L_{K_{2}K_{1}})\leq2\kappa^{K_{2}-K_{1}}R,\\
d_{\mathcal{P}}\big((T_{N}\text{\dots}T_{K_{2}}),L_{NK_{2}})\leq2\kappa^{N-K_{2}}R.
\end{cases}
\]
We then use Proposition \ref{prop:distanceProduit}, to obtain the
inequality for the partition function,
\[
d_{\mathcal{P}}(T_{N}\text{\dots}T_{0},L_{NK_{2}}L_{K_{2}K_{1}}L_{K_{1}0})\leq2R(\kappa^{N-K_{2}}+\kappa^{K_{2}-K_{1}}+\kappa^{K_{1}}),
\]
 for the density function
\begin{multline*}
d_{\mathcal{P}}(T_{N}\text{\dots}T_{K_{1}+1}XT_{K_{1}}\text{\dots}T_{0},L_{NK_{2}}L_{K_{2}K_{1}}XL_{K_{1}0})\leq2R(\kappa^{N-K_{2}}+\kappa^{K_{2}-K_{1}}+\kappa^{K_{1}})\\
d_{\mathcal{P}}(T_{N}\text{\dots}T_{K_{2}+1}YT_{K_{2}}\text{\dots}T_{0},L_{NK_{2}}YL_{K_{2}K_{1}}L_{K_{1}0})\leq2R(\kappa^{N-K_{2}}+\kappa^{K_{2}-K_{1}}+\kappa^{K_{1}}),
\end{multline*}
and the pair correlation function
\begin{align*}
 & d_{\mathcal{P}}(T_{N}\text{\dots}T_{K_{2}+1}YT_{K_{2}}\text{\dots}T_{K_{1}+1}XT_{K_{1}}\text{\dots}T_{0},L_{NK_{2}}YL_{K_{2}K_{1}}XL_{K_{1}0})\\
 & \qquad\leq2R(\kappa^{N-K_{2}}+\kappa^{K_{2}-K_{1}}+\kappa^{K_{1}}).
\end{align*}
Moreover we have 
\begin{align*}
 & (u,L_{NK_{2}}L_{K_{2}K_{1}}L_{K_{1}0}v)\cdot(u,L_{NK_{2}}YL_{K_{2}K_{1}}XL_{K_{1}0}v)\\
 & \qquad=(u,z_{NK_{2}})(l_{NK_{2}}z_{K_{2}K_{1}})(l_{K_{2}K_{1}}z_{K_{1}0})(l_{K_{1}0}v)(u,z_{NK_{2}})(l_{NK_{2}}Y(z_{K_{2}K_{1}}))\\
 & \qquad\quad(l_{K_{2}K_{1}}X(z_{K_{1}0}))(l_{K_{1}0}v)\\
 & \qquad=(u,z_{NK_{2}})(l_{NK_{2}}z_{K_{2}K_{1}})(l_{K_{2}K_{1}}X(z_{K_{1}0}))(l_{K_{1}0}v)(u,z_{NK_{2}})(l_{NK_{2}}Y(z_{K_{2}K_{1}}))\\
 & \qquad\quad(l_{K_{2}K_{1}}z_{K_{1}0})(l_{K_{1}0}v)\\
 & \qquad=(u,L_{NK_{2}}L_{K_{2}K_{1}}XL_{K_{1}0}v)\cdot(u,L_{NK_{2}}YL_{K_{2}K_{1}}L_{K_{1}0}v)
\end{align*}
and this allows us to conclude that 
\begin{align*}
 & Z^{2}\rho_{K_{1}}(X)\rho_{K_{2}}(Y)\\
 & \quad\leq(u,L_{NK_{2}}L_{K_{2}K_{1}}XL_{K_{1}0}v)\cdot(u,L_{NK_{2}}YL_{K_{2}K_{1}}L_{K_{1}0}v)e^{4R(\kappa^{N-K_{2}}+\kappa^{K_{2}-K_{1}}+\kappa^{K_{1}})}\\
 & \quad\leq(u,T_{N}\text{\dots}T_{K_{2}+1}XT_{K_{2}}\text{\dots}T_{K_{1}+1}XT_{K_{1}}\text{\dots}T_{0}v)\cdot(u,T_{N}\text{\dots}T_{0}v)\\
 & \quad\quad\times e^{8R(\kappa^{N-K_{2}}+\kappa^{K_{2}-K_{1}}+\kappa^{K_{1}})}\\
 & \quad\leq Z^{2}\rho_{K_{2},K_{1}}(Y,X)e^{8R(\kappa^{N-K_{2}}+\kappa^{K_{2}-K_{1}}+\kappa^{K_{1}})}.
\end{align*}
Finally, we have 
\[
Z^{2}\rho_{K_{1}}(X)\rho_{K_{2}}(Y)\geq Z^{2}\rho_{K_{2},K_{1}}(Y,X)e^{-8R(\kappa^{N-K_{2}}+\kappa^{K_{2}-K_{1}}+\kappa^{K_{1}})}.
\]

\end{proof}
The proof of the decay of the cluster correlation is the same. One
should just replace $X_{i}$ by $X_{i}=T_{i+l}Y_{i,l-1}\cdots T_{i+1}Y_{i,2}T_{i}Y_{i,1}$,
which are positive operators.

\subsection{Smoothness of the free energy \label{sec:Smoothness-of-the}}

In this section, we prove Proposition \ref{prop:NormEqui-1} and Theorem
\ref{thm:Smooth-1-1}.

\subsubsection{Proof of Proposition \ref{prop:NormEqui-1}}
\begin{proof}
Let $H$ a hyperplane such that $E=\text{Vect}(\text{\{}x_{0}\text{\}},H)$.
The projective space is locally isomorph to $H$. Let $B$ be the
convex set containing all the $s\in H$ for which there exist $(\alpha_{+},\alpha_{-},\beta_{-},\beta_{+})$
satisfying $\alpha_{+}x_{0}\leq x_{0}+s\leq\beta_{+}x_{0},\text{ with }\beta_{+}-\alpha_{+}\leq r$,
and $\alpha_{-}x_{0}\leq x_{0}-s\leq\beta_{-}x_{0},\text{ with }\beta_{-}-\alpha_{-}\leq r$,
for some $r$ small enough. This set is symmetric with respect to
the transformation $s\rightarrow-s$. Therefore, it is the ball of
the norm $\|s\|=r\cdot\inf(\lambda\in\mathbb{R},\frac{s}{\lambda}\in B))$.
Let us check that this norm is close to the distance. Let $s_{1},s_{2}\in H$
be such that $d(x_{0}+s_{1},x_{0})<r$ and $d(x_{0}+s_{2},x_{0})<r$.
We have then $\alpha_{1}x_{0}\leq x_{0}+s_{1}\leq\beta_{1}x_{0}$
and $\alpha_{2}x_{0}\leq x_{0}+s_{2}\leq\beta_{2}x_{0}$ and because
$r$ is very small, we can write $\alpha_{1}=1+\delta\alpha_{1}$,
$\alpha_{2}=1+\delta\alpha_{2}$, $\beta_{1}=1+\delta\beta_{1}$,
$\beta_{2}=1+\delta\beta_{2}$. At first order we have $d(x_{0}+s_{1},x_{0})=\delta\beta_{1}-\delta\alpha_{1}+o(|\delta\beta_{1}|,|\delta\alpha_{1}|)$
and $d(x_{0}+s_{2},x_{0})=\delta\beta_{2}-\delta\alpha_{2}+o(|\delta\beta_{2}|,|\delta\alpha_{2}|)$. 

We now check that $d(x_{0}+s_{1},x_{0})=(1+O(r))\|s_{1}\|$. First
we have 
\[
d(x_{0}+s_{1},x_{0})\geq(1+O(r))\|s_{1}\|.
\]
 Indeed, for any $\lambda\in\mathbb{R},$ we have $\lambda(1+\delta\alpha_{1})x_{0}+(1-\lambda)x_{0}\leq x_{0}+\lambda s\leq\lambda(1+\delta\beta_{1})x_{0}+(1-\lambda)x_{0}$
and $x_{0}+\lambda(\delta\alpha_{1})x_{0}\leq x_{0}+\lambda s\leq x_{0}+\lambda\delta\beta_{1}x_{0}$.
With $\lambda=\frac{r}{\delta\beta_{1}-\delta\alpha_{1}}$, we obtain
\[
x_{0}+\frac{r}{\delta\beta_{1}-\delta\alpha_{1}}(\delta\alpha_{1})x_{0}\leq x_{0}+\frac{r}{\delta\beta_{1}-\delta\alpha_{1}}s\leq x_{0}+\frac{r}{\delta\beta_{1}-\delta\alpha_{1}}\delta\beta_{1}x_{0}
\]
Therefore 
\[
\|s\|\leq\delta\beta_{1}-\delta\alpha_{1}=(1+O(r))d(x_{0},x_{0}+s).
\]
Then we claim that 
\[
d(x_{0}+s_{1},x_{0})\leq(1+O(r))\|s_{1}\|.
\]
Indeed let $\lambda$ be such that for any $\alpha,\beta$ $\alpha x_{0}\leq x_{0}+\frac{s}{\lambda}\leq\beta x_{0}$
$\Rightarrow\beta-\alpha\geq r$. Then for any $\alpha,\beta$ $\lambda\alpha x_{0}+(1-\lambda)x_{0}\leq x_{0}+s\leq\lambda\beta x_{0}+(1-\lambda)x_{0}$
$\Rightarrow\beta-\alpha\geq r$. Therefore
\begin{align*}
d(x_{0},x_{0}+s)\leq\log\frac{1-\lambda+\lambda\beta}{1-\lambda+\lambda\alpha} & =\log\frac{1+\lambda\delta\beta}{1+\lambda\delta\alpha}\\
 & =(\lambda(\delta\beta-\delta\alpha))(1+O(r))\leq\|s\|(1+O(r)).
\end{align*}
 We finally check that $d(x_{0}+s_{1},x_{0}+s_{2})=(1+O(r))d(x_{0}+s_{1}-s_{3},x_{0}+s_{2}-s_{3})$
for any $s_{1},s_{2},s_{3}\in B$. We have, $\alpha_{3}x_{0}\leq x_{0}+s_{3}\leq\beta_{3}x_{0}$,
and $d(x_{0}+s_{1},x_{0}+s_{2})=(\delta\beta-\delta\alpha)(1+O(r))$
with $(1+\delta\alpha)(x_{0}+s_{1})\leq x_{0}+s_{2}\leq(1+\delta\beta)(x_{0}+s_{1})$.
Then 
\[
(1+\delta\alpha+O(r)\delta\alpha)(x_{0}+s_{1}+s_{3})\leq(1+\delta\alpha)(x_{0}+s_{1}+s_{3})-\delta\alpha s_{3}\leq x_{0}+s_{2}+s_{3}
\]
and
\[
x_{0}+s_{2}+s_{3}\leq(1+\delta\beta)(x_{0}+s_{1}+s_{3})-\delta\alpha s_{3}\leq(1+\delta\beta+O(r))(x_{0}+s_{1}+s_{3}).
\]
We conclude that $d(x_{0}+s_{1},x_{0}+s_{2})=\|s_{2}-s_{1}\|(1+O(r))$.
\end{proof}
Let $\mathcal{C}$ be the cone of positive vectors in $\mathbb{R}^{n}$
and let $x_{0}=(x^{1},\cdots,x^{n})$ and $H=\text{\{}s:\sum_{i}s^{i}=0\text{\}}$.
Then in a neighborhood of $x_{0}$, we have 
\[
\alpha x_{0}\leq x_{0}+s\leq\beta x_{0}
\]
 with 
\[
\alpha=\max_{\alpha}\alpha x^{i}\leq x^{i}+s^{i}=1+\min\frac{s^{i}}{x^{i}}
\]
 and 
\[
\beta=\min_{\beta}\beta x^{i}\geq x^{i}+s^{i}=1+\max\frac{s^{i}}{x^{i}}.
\]
In addition 
\[
d(x_{0},x_{0}+s)=\log\frac{1+\max\frac{s^{i}}{x^{i}}}{1+\min\frac{s^{i}}{x^{i}}}\approx\max\frac{s^{i}}{x}-\min\frac{s^{i}}{x}=\max\frac{s^{i}}{x^{i}}+\max\frac{-s^{i}}{x^{i}}.
\]
Finally the constructed norm is then:\label{ExampleNorm}
\[
\|s_{1}-s_{2}\|=\max\frac{s_{1}^{i}-s_{2}^{i}}{x^{i}}+\max\frac{s_{2}^{i}-s_{2}^{i}}{x^{i}}.
\]
In order to prove Theorem \ref{thm:Smooth-1-1} we will need the following
lemma.
\begin{lem}
Let $E_{n}$ be Banach spaces with norms $\|\|_{n}$. Consider the
functions $u_{n}:\mathbb{R\rightarrow}E_{n}$ iteratively defined
by\label{LemCompositionFonctions}

\[
u_{n+1}(s)=g_{n}(s,u_{n}(s)),
\]
with $g_{n}(s,0)=0$ which are assumed to be uniformly contracting,
$\|\partial_{2}g_{n}\|_{\|\|_{n}\rightarrow\|\|_{n+1}}\leq\kappa$
with $\kappa<1$. If the $g_{n}$ are uniformly $C^{k}$ then the
$u_{n}$ are uniformly $C^{k}$.\end{lem}
\begin{proof}
We prove by induction that there exists constant a $C_{k}$ such that
for all $n$, $\|\frac{d^{k}}{ds^{k}}u_{n}\|_{n}\leq C_{k}$. Computing
the derivative gives

\[
\frac{d^{k}}{ds^{k}}u_{n+1}=\partial_{2}g_{n}\cdot\frac{d^{k}}{ds^{k}}u_{n}+Q\Big(g_{n},(\partial_{1}^{r}\partial_{2}^{p}g_{n}),(\frac{d^{i}}{ds^{i}}u_{n})_{i<k}\Big)
\]
where $Q$ is a polynomial involving lower order derivatives of $u_{n}$
and the derivatives of $g_{n}$. Because of the induction hypothesis,
there exists $C_{k-1}$ such that for all $n$ and all $i<k$, $\|\frac{d^{i}}{ds^{i}}u_{n}\|\leq C_{k-1}$.
Therefore $Q$ can be uniformly bounded by a constant $\tilde{C}_{k}$
which depends only on $\sup_{n\in\mathbb{N}}\|\partial_{1}^{r}\partial_{2}^{p}g_{n}\|$
and $C_{k-1}$. We have therefore 

\[
\|\frac{d^{k}}{ds^{k}}u_{n+1}\|_{n+1}\leq\kappa\|\frac{d^{k}}{ds^{k}}u_{n}\|_{n}+\tilde{C_{k}}
\]
and we can then set 
\begin{equation}
C_{k}=\frac{1}{1-\kappa}\tilde{C_{k}}.\label{eq:ConstanteItere}
\end{equation}
We can now conclude because if $g$ is contracting for $d$ then it
is contracting for $\|\|$.\end{proof}
\begin{example}
Consider $T(\beta)=\begin{pmatrix}\beta & \epsilon\\
\epsilon & 1
\end{pmatrix}$, with largest eigenvalue 
\[
\lambda(\beta)=\frac{(\beta+1)+\sqrt{(\beta-1)^{2}+4\epsilon^{2}}}{2}
\]
and $\log(\lambda(\beta))$ for $\beta$ around $1$. In the usual
positive cone, 
\[
\Delta(T(\beta)(\mathcal{C}))=d_{\mathcal{C}}\left(\begin{pmatrix}1\\
\epsilon
\end{pmatrix},\begin{pmatrix}\epsilon\\
1
\end{pmatrix}\right)=|2\log(\epsilon)|.
\]
 The Birkhoff-Hopf theorem gives $\kappa=\tanh(\log(\epsilon)/2)\approx1-2\epsilon$.
Around the point $\begin{pmatrix}1\\
1
\end{pmatrix}$, the norm $\mathcal{N}$ is equal to the norm $\|.\|_{\infty}$ (see
example \ref{ExampleNorm}). The iterative formula (\ref{eq:ConstanteItere})
gives a constant behave like $C_{k}\approx(2\epsilon)^{(1-k)}$ and
this is what we get with the exact calculation of $\frac{d}{d\beta^{k}}[\log(\lambda(\beta))]$.\end{example}
\begin{rem}
If $T$ is contracting for the distance $d$, then $T$ is locally
contracting for $\mathcal{N}$. 
\end{rem}
We can now finish the proof of Theorem \ref{thm:Smooth-1-1}.
\begin{proof}
{[}Theorem \ref{thm:Smooth-1-1}{]}We denote 
\[
u_{n}(\beta)=\frac{\prod_{i=0}^{n-1}T_{i}(\beta)b}{\|\prod_{i=0}^{n-1}T_{i}(\beta)\|}
\]
 and we decompose the log of the product as 
\[
f_{N}(\beta)=\frac{1}{N}\log\left\langle a,\prod_{i=0}^{N-1}T_{i}(\beta)b\right\rangle =\frac{1}{N}\log\langle a,u_{N}(\beta)\rangle+\frac{1}{N}\sum\log(\|T_{i}(\beta)u_{i}(\beta)\|).
\]
Because the $T_{i}$ are smooth we only have to make sure that the
$u_{i}$ are smooth as well. This follows from the previous lemma.
The function $u_{i}(\beta)$ is smooth for the constructed norm $\|\|_{i}$.
But because the cone is normal, 
\[
\alpha x\leq y\leq\beta x\Rightarrow\|y-\frac{\beta+\alpha}{2}x\|\leq\frac{\beta-\alpha}{2}\|x\|,
\]
 we have then that $\|x\|\cdot\|x-y\|_{i}\geq\|x-y\|$ and we conclude
because $\|u_{n}(\beta)\|=1$. \end{proof}
\begin{prop}
Let $g_{1},g_{2},\cdots,g_{n}\cdots$, be analytic functions such
that for any $k$, $g_{k}(x)=\sum_{i}b_{k,n}x^{n}$ with $|b_{k,n}|\leq r^{n}$
and $|b_{k,0}|\leq1$. Let $f_{0}=g_{0}$ and $f_{k+1}=(1+\kappa g_{k+1}f_{k})$.
Then for any $k$, 
\[
f_{k}=\sum c_{k,n}x^{n}
\]
with $c_{k,n}\leq d_{n}$ where $d_{n}$ are the coefficient of the
Taylor expansion of $\frac{1-rx}{(1-\kappa)-rx}$. In particular,
if $f_{n}$ admits a limit $f_{\infty}$, then $f_{\infty}$ is analytic.\end{prop}
\begin{proof}
We can assume $b_{k,n}=r^{n}$ for any $k,n$. Indeed another configuration
would give a smaller $c_{k,n}$. We expand $f_{k}$ and have: $f_{k}=\sum_{i=0}^{k}\big(\frac{\kappa}{1-rx}\big)^{i}$
whose coefficients are then smaller than those of $\sum_{i=0}^{\infty}\big(\frac{\kappa}{1-rx}\big)^{i}=\frac{1-rx}{(1-\kappa)-rx}.$\end{proof}
\begin{cor}
Let $g_{n}$ and $u_{n}$ be as in Proposition \ref{LemCompositionFonctions}.
Suppose that there exists $r\geq0$ such that $\frac{\|\partial_{s}^{i}g\|}{i\text{!}}\leq r^{i}$
for all $i$, then $u_{n}$ are analytic with coefficients of its
Taylor series bounded by that of $\frac{1-rx}{(1-\kappa)-rx}$. In
particular if $u_{n}$ admits a limit then it is analytic with convergence
radius $\frac{1-\kappa}{r}$.\end{cor}
\begin{proof}
This follows from the fact that 
\[
u_{n}(s)-u_{n}(0)=g_{n}(s,u_{n-1}(s)-u_{n-1}(0))+g_{n}(s,u_{n-1}(0))-u_{n}(0).
\]

\end{proof}

\subsection{Proof of Theorem \ref{Theorem(TCL)-1}}

\label{sect.TCL}We now prove the central limit theorem \ref{Theorem(TCL)-1}
from the regularity of the Laplace transform.
\begin{proof}
By Theorem \ref{TheoremAnalytic} $f(\alpha)$ is smooth with $\partial_{\alpha}f|_{\alpha=0}=\gamma\sqrt{N}\leq C\sqrt{N}$,
$\partial_{\alpha}^{2}[f-\gamma\sqrt{N}\alpha]|_{\alpha=0}=\sigma^{2}\leq C$
and $\partial_{\alpha}^{3}[f-\gamma\sqrt{N}\alpha]\leq\frac{C}{\sqrt{N}}.$
Then $(f(\alpha)-\gamma\alpha)=1+\frac{(\sigma\alpha)^{2}}{2}+O(\frac{1}{\sqrt{N}}).$
Therefore the Laplace transform is close to the one of a Gaussian
and we can conclude with the usual Berry Essen inequality. 
\end{proof}

\section{Proofs for the Jellium model}

\subsection{Proof for the classical Jellium model}

\label{sub:proofClass}We first write the partition function in the
form of products of operators. Recall that $U_{i}(s)=-2q_{i}\int_{\tilde{x}_{i}}^{\tilde{x_{i}}+s}(y-\tilde{x}_{i})\rho(y)dy$
with $\tilde{x_{i}}$ the equilibrium position of the particle $i$.
We note $\delta=\frac{1}{2}\min(|\tilde{a}_{i}-\tilde{a}_{i+1}|)$.
\begin{defn}
(Iterative operator) Let $T_{i}$ be the operator defined for any
function $f$ in $L^{1}$ or $L^{\infty}$ by 
\[
T_{i}f(x)=\int_{s=x-\tilde{x}_{i+1}+\tilde{x_{i}}}^{\infty}e^{-\beta U(s)}f(s)ds.
\]
In particular, we can rewrite the partition function as
\[
\mathcal{Z}_{N}(\beta)=e^{-\beta E(\tilde{x}_{1},\text{\dots},\tilde{x}_{N})}\left\langle 1_{x_{N}<L-\tilde{x}_{N}},\left(\prod_{i=0}^{N-1}T_{i}\right)1_{x_{1}>-L-\tilde{x}_{1}}\right\rangle 
\]

\end{defn}
We are then in the setting of Section \ref{sec:Decay-of-correlation}.

\subsubsection{Construction of a uniform invariant cone}

We first notice that we cannot directly apply the Birkhoff-Hopf Theorem
with the cone of positive functions $\mathcal{C}_{0}$. Indeed we
have the
\begin{rem}
For any $T_{i}$, $\Delta_{\mathcal{C}_{0}}(T_{i})=\infty$. For example
$\text{Supp}[T_{i}(1_{[0,1]})]=(-\infty;1+(\tilde{x}_{i+1}-\tilde{x}_{i})]$
and $\text{Supp}(T_{i}(1_{[2,3]}))=(-\infty,3+(\tilde{x}_{i+1}-\tilde{x}_{i})]$
and then we have for any $\alpha>0$ and $1+(\tilde{x}_{i+1}-\tilde{x}_{i})<t<3+(\tilde{x}_{i+1}-\tilde{x}_{i})$,
$\big(T_{i}(1_{[0,1]})-\alpha T_{i}(1_{[2,3]})\big)(t)<0$ so $\alpha_{\min}=0$. 
\end{rem}
The solution is to construct another cone. If we were restricted to
a bounded interval, then the simplest solution would be to consider
finite products of $T_{i}$, instead of one by one. The kernel of
$\prod_{i=n}^{n+k-1}T_{i}$ is strictly positive on $\text{\{}(x,y),y\geq x-2n\delta\text{\}}$,
and therefore with $n$ such that $2n\delta>2A$, the kernel is strictly
positive. $\prod T_{i}$ are then contracting for the cone $\text{\{}f\geq0\text{\}. }$

In our case, because of the multiplication by $e^{-U_{\text{i}}(s)}$,
we will be able to neglect the influence of $fe^{-U_{i}(s)}$ outside
$(-A,A)$. We choose $A$ such that 

\[
\int_{A}^{\infty}e^{-U_{i}(s)}ds\leq\frac{\delta}{2}e^{-U_{i}(A)}
\]
(for example, because of $\frac{d}{ds}U(s)\geq qms,$ we can choose
$A=\frac{2}{\delta qm}$). In Proposition \ref{prop:ConeConstruc}
we will define a cone such that $f$ can be slightly negative for
$\text{\{}x:x\geq A\text{\}}$. We also make it so that $T_{i}$ are
contracting and not only $\prod_{k}^{k+n}T_{i}$. The price to pay
is more restrictions. Intuitively, it is how $\prod T_{i}f$ looks
like for $f\geq0$.

Let us divide the interval $[-A,A]$ in small intervals with $I_{k}=[k\delta/2,(k+1)\delta/2]\text{ with }k\in\mathbb{\mathbb{Z}}\text{ and }-2\frac{A}{\delta}-1=k_{min}\leq k\leq k_{max}=2\frac{A}{\delta}+1$.
We suppose $\frac{A}{\delta}\in\mathbb{N}$ to simplify the notation. 
\begin{prop}
There exist $(\epsilon_{k})_{-2\frac{A}{\delta}\leq k\leq2\frac{A}{\delta}}$
such that the cone $\mathcal{C}$ defined by\label{prop:ConeConstruc}
\[
f\in\mathcal{C}\Leftrightarrow\begin{cases}
\forall t\geq A & f(t)+\epsilon I_{k_{max}}(f)\geq0,\\
\forall t\leq A & f(t)\geq0,\\
\text{on }-A\leq t\leq A & f\text{ is decreasing },\\
\forall t\leq-A & f(t)\geq f(-A),\\
\forall k\in[-2\frac{A}{\delta},2\frac{A}{\delta}] & I_{k-1}(f)\leq\frac{1}{\epsilon_{k}}I_{k}(f),\\
\forall t\leq-A & f(t)\leq\frac{1}{\epsilon'}I_{k_{min}}(f),
\end{cases}
\]
 satisfies that, for any $i$, $T_{i}$ is $d_{\mathcal{C}}$ contracting. 
\end{prop}
This cone may seem a bit artificial, however it behaves nicely with
respect to the iteration of $T_{i}$. For the proof we need the following 
\begin{lem}
Let $y,x>0$, $K$ linear and $a,b,u,v\geq0$ such that $Kx\ge ax+uy$
and $Ky\leq by+vx$. If $a>b$ or $u>0$, then there exist $\epsilon>0$
such that if $\frac{1}{\epsilon}x\geq y$ then $\frac{1}{\epsilon}Kx>Ky$.
\label{lem:TechniEps}\end{lem}
\begin{proof}
If $a>b$, then for $\epsilon$ small enough, $\frac{b}{\epsilon}+v<\frac{a}{\epsilon}$
and we have $Ky\leq by+vx\leq(\frac{b}{\epsilon}+v)x<\frac{a}{\epsilon}x\leq\frac{1}{\epsilon}Kx$.
If $u>0,$ then we have $\frac{1}{\epsilon}Kx-Ky\geq(\frac{a}{\epsilon}-v)x-(b-\frac{u}{\epsilon})y>0$
for $\epsilon$ small enough.
\end{proof}
We can now carry on the proof of Proposition \ref{prop:ConeConstruc}.
\begin{proof}
We construct the $\epsilon_{k}$ recursively. Because $f$ is decreasing
and $e^{-U_{i}}$ are uniformly integrable, there exists $u$ such
that $\int_{-\alpha}^{\infty}fe^{-U_{i}(s)}ds\leq uI_{k_{min}}(f)$.
We then have 
\[
\sup T_{i}f=\int_{\mathbb{R}}e^{-U_{i}(s)}f(s)ds\leq\sup f\cdot\int_{-\infty}^{-A}e^{-U_{i}(s)}ds+uI_{k_{min}}(f).
\]
Moreover $I_{k_{min}}(Tf)\geq\delta e^{-U_{i}(-A)}I_{k_{min}}(Tf)$
and then $\int_{-\infty}^{-A}e^{-U_{i}(s)}ds<\delta e^{-U_{i}(-A)}$.
By Lemma \ref{lem:TechniEps} there exists $\epsilon'$ such that
for any $i$ and $t$ $,T_{i}f(t)<\frac{1}{\epsilon'}I_{k_{min}}(T_{i}f)$. 

Suppose we have constructed every $\epsilon_{k}$ up to $k=k_{0}$,
and let us construct $\epsilon_{k_{0}+1}$. Because of the induction
hypothesis there exist $b_{k_{0}}$ such that $\sup_{t}f(t)\leq b_{k_{0}}I_{k_{0}}(f)$
so there exists $b_{k_{0}}'$ such that $I_{k_{0}}(T_{i}f)\leq b_{k_{0}}'I_{k_{0}}(f)$.
Moreover

\begin{align*}
\forall a\in I_{k_{0}+1},\quad T_{i}f(a) & =\int_{a-\tilde{x}_{i+1}-\tilde{x}_{i}}^{\infty}f(s)e^{-U_{i}(s)}ds\\
 & \geq\int_{I_{k_{0}}}f(s)e^{-\max_{s\in I_{k_{0}}}U_{i}(s)}ds\\
 & \geq e^{-\max_{s\in I_{k_{0}}}U_{i}(s)}I_{k_{0}}(f).
\end{align*}
So thanks to Lemma \ref{lem:TechniEps}, there exists $\epsilon_{k_{0}+1}$
such that if for all $k\leq k_{0},I_{k}(f)\leq\frac{1}{\epsilon_{k+1}}I_{k+1}(f)$
then for all $k\leq k_{0}+1,I_{k}(T_{i}f)<\frac{1}{\epsilon_{k+1}}I_{k_{0}}(T_{i}f)$.
We also have for any $a\leq A$

$\begin{alignedat}{1}T_{i}f(a) & =\int_{a-\tilde{x}_{i+1}-\tilde{x}_{i}}^{\infty}f(s)e^{-U(s)}ds\\
 & \geq\int_{A-\delta}^{A}f(s)e^{-U_{i}(s)}ds-\epsilon\int_{A}^{\infty}e^{-U_{i}(s)}ds\cdot I_{k_{max}}(f)\\
 & \geq e^{-U_{i}(A)}(1-\epsilon\frac{\delta}{2})I_{k_{max}}(f).
\end{alignedat}
$ 

In particular $T_{i}f(a)\geq0$. Moreover for any $a>A$, we have
\begin{align*}
T_{i}f(a)\geq-\epsilon\int_{A}^{\infty}e^{-U_{i}(s)}ds\cdot I_{k_{max}}(f) & \geq-\epsilon\frac{\delta}{2}e^{-U_{i}(A)}I_{k_{max}}(f)\\
 & \geq-\epsilon\frac{1}{1-\epsilon\frac{\delta}{2}}I_{k_{max}}(Kf).
\end{align*}
Because $f\geq0$ on $(-\infty,A]$, $Tf$ is decreasing on $]-\infty,A+\delta]$.
To conclude, it will be enough to compare $T_{i}f$ with $T_{i}g$
for $f,g\in\mathcal{C}$ and $I_{k_{max}}(f)=I_{k_{max}}(g)=1$. Because
all the inequalities become strict, there exists $\epsilon''$ such
that for any $f\in\mathcal{C}$ with $I_{k_{max}}(f)=1$, if $\|g\|_{L^{\infty}}\leq\epsilon''$
then $T_{i}(f-g)\in\mathcal{C}$. Moreover for any $g\in\mathcal{C}$
with $I_{k_{max}}(g)=1$, $\|g\|_{L^{\infty}}\leq\prod\frac{1}{\epsilon_{k}}$.
So $T_{i}(f-\epsilon''\prod\epsilon_{k}g)\in\mathcal{C}$. And this
concludes the proof because then $\Delta\leq2\log(\epsilon''\prod\epsilon_{k})$.\end{proof}
\begin{rem}
If we denote by $\mathcal{H}_{k}$ the assertion $"I_{k-1}(f)\leq\frac{1}{\epsilon_{k}}I_{k}(f)"$,
we have actually proved that if $f$ satisfies all the condition of
$\mathcal{C}$ except $(\mathcal{H}_{i})_{i=r,\text{\dots},k_{max}}$,
then $T_{i}f$ satisfies all the condition of $\mathcal{C}$ except
$(\mathcal{H}_{i})_{i=r+1,\text{\dots},k_{max}}$. This implies that
if $f\geq0$ and $\text{supp}(f)\subset[-A,A]$, then $\prod_{i=k}^{k+n}T_{i}f\in\mathcal{C}$
is as wanted.
\end{rem}

\subsubsection{Decay of correlation.}

Here we prove Theorem \ref{Theo-decay-correlation-Jelium-Classic}.

We can carry on with the construction of conditions like $I_{k-1}(f)\leq\frac{1}{\epsilon_{k}}I_{k}(f)$
after $k_{max}$ in Proposition \ref{prop:ConeConstruc}. We denote
by $\mathcal{C}_{m}$ this more specified cone replacing $\forall k\in[-2\frac{A}{\delta},2\frac{A}{\delta}]I_{k-1}(f)\leq\frac{1}{\epsilon_{k}}I_{k}(f)$
by $\forall k\in[-2\frac{A}{\delta},2\frac{A}{\delta}+m]I_{k-1}(f)\leq\frac{1}{\epsilon_{k}}I_{k}(f)$.
We have then the 
\begin{prop}
If $g\in\mathcal{C}$ then $\prod_{i=k}^{k+n}T_{i}g\in\mathcal{C}_{i}.$\end{prop}
\begin{proof}
By Proposition \ref{prop:ConeConstruc}, $\prod_{i=k}^{k+n}T_{i}g\in\mathcal{C}$.
Therefore it is enough to prove the remaining conditions. Let $g_{i}\in\mathcal{C}_{i}$.
As previously,

\[
T_{i}g(a)\geq-\epsilon\frac{\delta}{2}e^{-U_{i}(A+\delta i)}I_{k_{max}}(f)\geq-\epsilon\frac{1}{1-\epsilon\frac{\delta}{2}}I_{k_{max}+i}(Tf).
\]
As a consequence we have that for $f\geq0$ and $\text{supp}(f)\in(-\infty,A+n\delta]$
then $f\prod_{i=k}^{k+n}T_{i}g\geq0$ for all $g\in\mathcal{C}$. \end{proof}
\begin{rem}
\label{Remark-:-iterPositif}If $f\geq0$ and $\text{supp}(f)\in[A,\infty)$
then $T_{i}f\in\mathcal{C}$.Therefore for $f$, $\text{supp}(f)\in[a-dl,a-dl]$
$dl\leq\delta$ and $A\leq a\leq A+n\delta$ then $T_{n+k+1}F\prod_{i=k}^{k+n}T_{i}$
is order preserving for the cone $\mathcal{C}$. \end{rem}
\begin{prop}
For $f\geq0$ with $\text{supp}(f)\in[a-dl,a-dl]$ $dl\leq\delta$
and $a\geq A+n\delta$, we have 
\[
\left\langle 1,\prod_{i=k}^{k+n}T_{i}f\right\rangle \leq e^{-n(|\frac{a}{\delta}|^{2}-c)}\langle1,f\rangle\left\langle 1,\prod_{i=k}^{k+n}T_{i}1\right\rangle .
\]
By iteration $\text{supp}(\prod_{i}Tf)\in[a-n\max(\tilde{a}_{i}-\tilde{a}_{i+1}),\infty)$.
Therefore
\begin{align*}
 & \|\prod T_{i}f\|_{\infty}\leq\langle1,f\rangle e^{-\sum\min_{i}(U_{i}(a-k(\max(\tilde{a}_{i}-\tilde{a}_{i+1}))))}\\
 & \qquad=\langle1,f\rangle e^{-\gamma_{min}\sum(a-k(\max(\tilde{a}_{i}-\tilde{a}_{i+1}))^{2}}.
\end{align*}
\end{prop}
\begin{proof}
There exists $\lambda>0$ such that $\langle1,\prod_{i=k}^{k+n-1}T_{i}1\rangle\geq\lambda^{n}$,
and we set $c=\log\lambda$. The result follows.
\end{proof}
We are now ready to prove the decay of the correlation functions.
Recall that for $i_{1},i_{2},\cdots,i_{k}$, we have the $k$-th marginal
defined by

\begin{align*}
 & \rho_{k}(x_{i_{1}},x_{i_{2}},\text{\dots},x_{i_{k}})=\frac{1}{\mathcal{Z}_{N}(\beta)}e^{-\beta E(\tilde{x}_{1},\text{\dots},\tilde{x}_{N})}\\
 & \quad\idotsint_{-L<x_{1}<x_{2}<\text{\dots}<x_{N}<L}\prod e^{-2\beta q_{i}\int_{\tilde{x}_{i}}^{x_{i}}\rho(y)(y-x_{i})dy}\prod_{i\neq i_{1},i_{2},...,i_{N}}dx_{i}.
\end{align*}
Note that there exists $r$ such that for $\rho_{k}(x_{i_{1}},x_{i_{2}},\text{\dots},x_{i_{k}})\leq e^{-r\max|x|^{3}}.$
\begin{cor}
There exists $\kappa<1$ and $C_{k}>0$ such that 

\[
|\rho_{k}(x_{i_{1}},x_{i_{2}},\text{\dots},x_{i_{k}})-\prod_{l=1}^{k}\rho_{1}(x_{i_{l}})|\leq C_{k}\kappa^{\inf|i_{l}-i_{l+1}|}.
\]
\end{cor}
\begin{proof}
Let $x_{i_{1}},\cdots,x_{i_{k}}\in\mathbb{R}^{k}$ and let $\delta^{(n)}$
be an approximation of the Dirac $\delta_{0}.$ We evaluate

\begin{align*}
 & \iiiint[\rho_{k}(y_{i_{1}},y_{i_{2}},\text{\dots},y_{i_{k}})-\prod\rho_{1}(y_{i_{1}})]\prod\delta^{(n)}(y_{i_{k}}-x_{i_{k}})dy_{i_{k}}\\
 & \quad=\iiiint[\rho_{k}(y_{i_{1}},y_{i_{2}},\text{\dots},y_{i_{k}})-\prod\rho_{1}(y_{i_{1}})]\prod\delta_{x_{i_{k}}}^{(n)}(y_{i_{k}})dy_{i_{k}}
\end{align*}
 which, in our formalism, is equal to

\begin{align*}
 & \frac{1}{\mathcal{Z}_{N}(\beta)}(u,T_{N}\text{\dots}T_{K_{k}+1}\delta_{x_{k}}^{(n)}T_{K_{k}}\text{\dots}T_{K_{1}+1}\delta_{1}^{(n)}T_{K_{1}}\text{\dots}T_{0}v)\\
 & \quad-\prod\frac{1}{\text{Z}}(u,T_{N}\text{\dots}T_{K_{1}+1}\delta_{x_{k}}^{(n)}T_{K_{1}}\text{\dots}T_{0}v).
\end{align*}

To begin with, assume $-A<x_{1},\cdots,x_{k}<A$. Because of Remark
\ref{Remark-:-iterPositif}, for any $i\in[1,k],$ $T_{K_{k}+m}\text{\dots}T_{K_{k}+1}\delta_{x_{k}}^{(n)}T_{K_{k}}$
is a positive operator. Changing $C_{k}$, we can suppose $\inf(|i_{l}-i_{l+1}|)>m$.
Therefore, denoting $X_{l}=T_{K_{l}+m}\text{\dots}T_{K_{l}+1}\delta_{x_{l}}^{(n)}T_{K_{l}}$
, we can apply Theorem \ref{thm:TwoPointCorelation} and we obtain
:

\begin{align*}
 & \big|\frac{1}{\mathcal{Z}_{N}(\beta)}(u,T_{N}\text{\dots}T_{K_{k}+1}\delta_{x_{k}}^{(n)}T_{K_{k}}\text{\dots}T_{K_{1}+1}\delta_{1}^{(n)}T_{K_{1}}\text{\dots}T_{0}v)\\
 & -\prod\frac{1}{\text{Z}}(u,T_{N}\text{\dots}T_{K_{1}+1}\delta_{x_{k}}^{(n)}T_{K_{1}}\text{\dots}T_{0}v)\big|\\
 & \quad\leq C_{k}\kappa^{\inf(i_{l}-i_{l+1})},
\end{align*}
 where $C_{k}=2k(2k+1)R\kappa^{-m}$ if $\inf(i_{l}-i_{l+1})$ is
larger that a constant $c$. Suppose that there exist $|x_{i}|>\epsilon\inf(|i_{l}-i_{l+1}|\delta)$
then $\rho_{k}(x_{i_{1}},x_{i_{2}},\text{\dots},x_{i_{k}})\leq e^{-r(\epsilon\delta)^{2}\inf_{l}|i_{l}-i_{l+1}|^{2}}$
and we are done. If for all $i,|x_{i}|\leq\epsilon\inf(|i_{l}-i_{l+1}|\delta)$,
then as previously $T_{K_{k}+m}\text{\dots}T_{K_{k}+1}\delta_{x_{k}}^{(n)}T_{K_{k}}$
is positive for $m=\epsilon\inf|i_{l}-i_{l+1}|$, hence 
\begin{align*}
 & |\frac{1}{\mathcal{Z}_{N}(\beta)}(u,T_{N}\text{\dots}T_{K_{k}+1}\delta_{x_{k}}^{(n)}T_{K_{k}}\text{\dots}T_{K_{1}+1}\delta_{1}^{(n)}T_{K_{1}}\text{\dots}T_{0}v)\\
 & -\prod\frac{1}{\text{Z}}(u,T_{N}\text{\dots}T_{K_{1}+1}\delta_{x_{k}}^{(n)}T_{K_{1}}\text{\dots}T_{0}v)|\\
 & \quad\leq C'_{k}\kappa^{(1-\epsilon)\inf(i_{l}-i_{l+1})}
\end{align*}
with $C_{k}=2k(2k+1)R$. We can conclude replacing $\kappa$ by $\kappa^{(1-\epsilon)}$.
\end{proof}

\subsubsection{Smoothness of the free energy for the classical Jellium model.}

Now we use Theorem \ref{thm:Smooth-1-1} to prove the smoothness of
the free energy. We first have to check its hypothesis. This is the
aim of the following proposition
\begin{prop}
\label{prop:SmoothProof}Let $u_{0}\in T_{i}(\mathcal{C})$. Then
\[
\|[\partial_{\beta}^{n}T_{i}(\beta)]\|_{\mathcal{N}_{i}\rightarrow\mathcal{N}_{i+1}}<\infty
\]
for all $n$, where $\mathcal{N}_{i}$ and $\mathcal{N}_{i+1}$ are
the norm constructed in Proposition \ref{prop:NormEqui-1} around
$u_{0}$ and $T_{i}(\text{\ensuremath{\beta}})u_{0}$ respectively,
and
\[
\mathcal{N}(\partial_{\beta}^{n}T_{i}(\beta)u_{0})<\infty.
\]

\end{prop}
To simplify the calculation, we introduce an approximating norm.We
define 
\[
a_{1}(s)=\sup_{t\geq A}\frac{s(t)+\epsilon I_{k_{max}}(s)}{\nu},
\]
\[
a_{2}(s)=\sup_{t\leq A}\frac{s(t)}{\nu},
\]
\[
a_{3}(s)=\sup_{-A\leq t\leq A}\frac{s'(t)}{\nu},
\]
\[
a_{4}(s)=\sup_{t\leq-A}\frac{s(t)-s(-A)}{\nu},
\]
\[
a_{5}(s)=\sup_{k}\frac{\frac{1}{\epsilon_{k}}I_{k}(s)-I_{k-1}(s)}{\nu},
\]
\[
a_{6}(s)=\sup_{t\leq-A}\frac{s(t)-\frac{1}{\epsilon'}I_{k_{min}}(s)}{\nu},
\]
 and we set $A_{\nu}(s)=\max(a_{1}(s),a_{2}(s),a_{3}(s),a_{4}(s),a_{5}(s),a_{6}(s))$.
Finally we define $\|s\|_{\nu}=\max(A(-s),A(s))$.

\begin{prop}
For $u_{0}$ in $T_{i}(\mathcal{C})$, there exists $\nu>0$ such
that 
\[
\|.\|_{i+1}\leq\|.\|_{\nu}.
\]
\end{prop}
\begin{proof}
Let $u_{0}\in T_{i}(\mathcal{C})$ We calculate the norm of Proposition
\ref{prop:NormEqui-1}. Let $s$ be such that

\[
\alpha u_{0}\leq u_{0}+s\leq\beta u_{0}
\]
with 
\[
\alpha=\max_{\alpha}\begin{cases}
\forall t\geq A & s(t)+\epsilon I_{k_{max}}(s)\geq(\alpha-1)[u_{0}(t)+\epsilon I_{k_{max}}(t)],\\
\forall t\leq A & s(t)\geq(\alpha-1)u_{0}(t),\\
\text{on }-A\leq t\leq A & s'(t)\leq(\alpha-1)u_{0}'(t),\\
\forall t\leq-A & s(t)-s(-A)\geq(\alpha-1)[u_{0}(t)-u_{0}(-A)],\\
\forall k\in[-2\frac{A}{\delta},2\frac{A}{\delta}] & I_{k-1}(s)-\frac{1}{\epsilon_{k}}I_{k}(s)\leq(1-\alpha)[I_{k-1}(u_{0})-\frac{1}{\epsilon_{k}}I_{k}(u_{0})],\\
\forall t\leq-A & s(t)-\frac{1}{\epsilon'}I_{k_{min}}(f)\leq(1-\alpha)[u_{0}(t)-\frac{1}{\epsilon'}I_{k_{min}}(u_{0})].
\end{cases}
\]
Because $u_{0}\in T_{i}(\mathcal{C}),$ there exists $\epsilon_{0}>0$
such that all the functions depending on $u_{0}$ on the right side
of the equation ($[u_{0}(t)+\epsilon I_{k_{max}}(t)],u_{0}(t),\cdots$)
can be bounded by $\epsilon_{0}$. We obtain

\[
\alpha'=\max_{\alpha'}\begin{cases}
\forall t\geq A & s(t)+\epsilon I_{k_{max}}(s)\geq(\alpha'-1)\epsilon_{0}\\
\forall t\leq A & s(t)\geq(\alpha'-1)\epsilon_{0}\\
\text{on }-A\leq t\leq A & s'(t)\leq(\alpha'-1)\epsilon_{0}\\
\forall t\leq-A & s(t)-s(-A)\geq(\alpha'-1)\epsilon_{0}\\
\forall k\in[-2\frac{A}{\delta},2\frac{A}{\delta}] & I_{k-1}(s)-\frac{1}{\epsilon_{k}}I_{k}(s)\leq(1-\alpha')\epsilon_{0}\\
\forall t\leq-A & s(t)-\frac{1}{\epsilon'}I_{k_{min}}(s)\leq(1-\alpha')\epsilon_{0}
\end{cases}
\]
and we then have $\alpha'\leq\alpha$. We have constructed then $\|\|_{\nu}$
with $\nu=\epsilon_{0}$.
\end{proof}

We finish the proof of Proposition \ref{prop:SmoothProof}.
\begin{proof}
We can now calculate $\|\partial_{\beta}^{n}T_{i}(\beta)\|_{\|\|_{i}\rightarrow\|\|_{\nu}}$.
Let $w$ with $\|w\|_{i}\leq1$, In particular, there exists $r>0$
such that $d(u_{0},u_{0}+r.w)\leq2r$. Therefore 

\[
(1-2r)u_{0}(t)\leq u_{0}+rw(t)\leq(1+2r)u_{0}(t)
\]
for all $t<A.$ Hence $|I_{k_{max}}(r.w)|\leq2r.I_{k_{max}}(u_{0})$
and for all $t>A$: 
\[
(1-2r)[u_{0}(t)]-4r.I_{k_{max}}(u_{0})\leq u_{0}(t)+rw(t)\leq(1+2r)u_{0}(t)+4r.I_{k_{max}}(u_{0}).
\]
Therefore

\[
[\partial_{\beta}^{n}T_{i}(\beta)w](x)=\int_{x-\tilde{x}_{i}-\tilde{x}_{i+1}}^{\infty}U_{i}(y)^{n}e^{-\beta U_{i}(y)}w(y)dy
\]
and there exist $c_{1},c_{2}$ and $c_{3}$ such that 
\[
\|[\partial_{\beta}^{n}T_{i}(\beta)w]\|_{L^{\infty}}\leq c_{1}\|u_{0}\|_{L^{\infty}},
\]
 
\[
\|[\partial_{\beta}^{n}T_{i}(\beta)w]\|_{L^{\infty}}\leq c_{2}\|u_{0}\|_{L^{\infty}}
\]
 and 
\[
\|[\partial_{\beta}^{n}T_{i}(\beta)w]'\|_{[-A,A]}\leq c_{3}\|u_{0}\|_{L^{\infty}}.
\]
There exists then $c'$ such that $\|[\partial_{\beta}^{n}T_{i}(\beta)w]\|_{\nu}\leq c'\|u_{0}\|_{L^{\infty}}$.
Then we have shown that $\|[\partial_{\beta}^{n}T_{i}(\beta)w]\|_{L^{\infty}}$
is a uniformly bounded operator for $\|\|_{i}\rightarrow\|\|_{\nu}$
and then for $\|\|_{i}\rightarrow\|\|_{i+1}.$We can now conclude
the proof of Theorem \ref{Theo-Smooth-Jelium-Classical}. Moreover,
\begin{align*}
\int_{x-\tilde{x}_{i}-\tilde{x}_{i+1}}^{\infty}U_{i}(y)^{n}e^{-\beta U_{i}(y)}w(y)dy & \leq\int_{x-\tilde{x}_{i}-\tilde{x}_{i+1}}^{\infty}(Ay^{2}){}^{n}e^{-\beta ay^{2}}w(y)dy\\
 & \leq\|w\|_{L^{\infty}}(\frac{A}{\beta a})^{n}\Gamma(n-\frac{1}{2})
\end{align*}
We can then apply Theorem \ref{TheoremAnalytic} which ends the proof
of analyticity of the free energy of the classical Jellium model.
\end{proof}

\subsection{Proof for the quantum Jellium model \label{sec:Quantique}}

\subsubsection{Decay of correlations, smoothness of the free energy.}
\begin{defn}
For any $f$ we define
\[
T_{i}f(\gamma)=\int_{E}1_{\forall t,\gamma(t)<\eta(t)+\delta_{i}}f(\eta)\nu_{i}(d\eta),
\]
where $\nu_{i}=\frac{1}{c}\int_{\mathbb{R}}\nu_{i,xx}dx$ with Radon
Nikodym density $\frac{d\nu_{i,xx}}{d\mu_{xx}}(\gamma)=e^{-\int_{0}^{\beta}U_{i}(\gamma(t))dt}$.
\end{defn}
We recall the result concerning the homogeneous case.
\begin{thm}
$T_{i}$ is a compact operator on $L^{1}((1+x^{2})^{-1}dx)$ with
a unique largest eigenvalue $\lambda_{M}>0$.
\end{thm}
For the reader's convenience, we have written again the proof. 
\begin{proof}
$T_{i}$ is a compact operator. Indeed let $u\in L^{1}((1+x^{2})^{-1}dx)$,
then $Tu$ is bounded, with finite variation. $T$ is Hilbert-Smicht:
\[
\iint_{E\times E}e^{-2\int U(\gamma(t))dt}1_{\forall t,\gamma(t)\leq\eta(t)+\delta_{i}}d\nu(\gamma)d\nu(\eta)<\infty.
\]
We have to check that $\nu_{i}$ are bounded measures $\mathbb{P}(|B_{t}|>y)\leq2e^{-|y|^{2}}$.
Then
\begin{align*}
\iint_{\mathbb{R\times}\Gamma}e^{-\int U_{i}(\gamma)dt}d\mu_{xx}(\gamma)dx & \leq\int_{\mathbb{R}}[\mu_{xx}(\inf_{t}\gamma(t)<\frac{x}{2})]+e^{-\beta U_{i}(\frac{x}{2})}dx\\
 & \leq\int_{\mathbb{R}}c[e^{-(\frac{x}{2})^{2}}+e^{-\beta U(\frac{x}{2})}]dx<\infty
\end{align*}

The largest eigenvalue is unique because the operator is irreducible
and we can apply the Krein Rutmann Theorem. \end{proof}
\begin{thm}
\label{thm:ConeLargerEigenvector}Let $T$ be a bounded real operator
whose spectral radius $\rho(T)$ is a non degenerate eigenvalue with
eigenvector $u_{0}$. Assume in addition that $T'=T-\rho(T)u_{0}u_{0}^{*}$
has a spectral radius $\rho(T')<\rho(T).$ Then there exists a cone
such that the operator is contracting.\end{thm}
\begin{proof}
We can suppose that the largest eigenvalue is $1$ and let $u_{0}$
be its eigenvector. We construct 
\[
\mathcal{C}=\text{positive linear combinations of }\cup_{n}T^{n}\Big(B(u_{0},\epsilon(1-\epsilon_{n}))\Big)
\]
 with $\epsilon_{n}$ a strictly decreasing sequence. Because there
exists $N$ such that $(T-u_{0}u_{0}^{*})^{n}$ is contracting for
$n\geq N$, 
\[
\mathcal{C}=\text{positive linear combinations of }\cup_{n\leq N_{1}}T^{n}\Big(B(u_{0},\epsilon(1-\epsilon_{n}))\Big).
\]

Indeed, let $x\in B(u_{0},\epsilon(1-\epsilon_{N_{1}+1}))$, $x=u_{0}+y+su_{0}$
with $u_{0}^{*}y=0$ and $\|y\|\leq c\|\epsilon\|$ and $s\leq c\|\epsilon\|$.
There exists $N_{1}$ such that $\|(T-u_{0}u_{0}^{*})^{N_{1}+1}\|\leq\frac{(1-\epsilon_{0})}{2c}$.
Then
\begin{align*}
T^{N_{1}+1}(x) & =(u_{0}u_{0}^{*})x+(T-u_{0}u_{0}^{*})^{N_{1}+1}(x)\\
 & =(1+s)u_{0}+(T-u_{0}u_{0}^{*})^{N_{1}+1}(y)\in B((1+s)u_{0},(1+s)\epsilon(1-\epsilon_{0})),
\end{align*}
 because $\|(T-u_{0}u_{0}^{*})^{N_{1}+1}(y)\|\leq\frac{\epsilon c(1-\epsilon_{0})}{2c}\leq(1+s)\epsilon(1-\epsilon_{0})$.

Let $x\in\mathcal{C}$, with $u_{0}^{*}x=1$. Then $x=\sum_{i=0}^{N_{1}}a_{i}x_{i},$
$a_{i}\geq0$ with $x_{i}\in T^{i}(B(u_{0},\epsilon(1-\epsilon_{i})))$.
Let us construct $\alpha$ and $\beta$ such that $\alpha u_{0}\leq T(x)\leq\beta u_{0}.$

First because $B(u_{0},\epsilon(1-\epsilon_{0}))\subset\mathcal{C}$,
we choose $\beta\leq\frac{\|T(x)\|}{\epsilon(1-\epsilon_{0})}$, and
we immediatly have $u_{0}-\frac{1}{\beta}T(x)\geq0.$

Second for all $i$, $T(x_{i})\in T^{i+1}(B(u_{0},\epsilon(1-\epsilon_{i})))$
and therefore $T(\frac{(1-\epsilon_{i+1})}{(1-\epsilon_{i})}x_{i})\in T^{i+1}(B(\frac{(1-\epsilon_{i+1})}{(1-\epsilon_{i})}u_{0},\epsilon(1-\epsilon_{i+1})))$.
We have then 
\[
T(\frac{(1-\epsilon_{i+1})}{(1-\epsilon_{i})}x_{i})-(1-\frac{(1-\epsilon_{i+1})}{(1-\epsilon_{i})})u_{0}\in T^{i+1}(B(u_{0},\epsilon(1-\epsilon_{i+1})))
\]
and also

\[
T(\frac{(1-\epsilon_{i+1})}{(1-\epsilon_{i})}x_{i})\geq(1-\frac{(1-\epsilon_{i+1})}{(1-\epsilon_{i})})u_{0}.
\]
So with $M=\max\frac{(1-\epsilon_{i+1})}{(1-\epsilon_{i})}$ and $m=\min\frac{(1-\epsilon_{i+1})}{(1-\epsilon_{i})}$
we have 
\[
MT(x)\geq m(\sum a_{i})u_{0}
\]
and we can conclude that 
\[
\frac{m\sum a_{i}}{M}u_{0}\leq T(x).
\]

\end{proof}
Such a construction is stable under small compact perturbations. 
\begin{prop}
There exists $\delta_{0}>0$ such that for $\delta T$ compact operator
with $\|\delta T\|\leq\delta_{0}$, 
\[
\Delta_{\mathcal{C}}((T+\delta T)(\mathcal{C}))<2\Delta_{\mathcal{C}}(T(\mathcal{C}))
\]
where $\mathcal{C}$ is the cone constructed in Theorem \ref{thm:ConeLargerEigenvector}

In particular $T+\delta T$ is a positive contracting operator for
$d_{\mathcal{C}}.$\end{prop}
\begin{proof}
We rewrite the proof of Theorem \ref{thm:ConeLargerEigenvector}.
We keep $u_{0}$ because $(T+\delta T)(u_{0})\in B(u_{0},(\epsilon(1-\epsilon_{0})))$.
First, it is enough to change $\beta\leq\frac{\|(T+\delta T)(x)\|}{\epsilon(1-\epsilon_{0})}$
and we have $u_{0}-\frac{1}{\beta}(T+\delta T)(x)\geq0$. Second,
we also have 
\begin{align*}
 & (T+\delta T)(\frac{(1-\epsilon_{i+1})}{(1-\epsilon_{i})}x_{i})-\delta T(\frac{(1-\epsilon_{i+1})}{(1-\epsilon_{i})}x_{i})\\
 & \quad\geq(T+\delta T)(\frac{(1-\epsilon_{i+1})}{(1-\epsilon_{i})}x_{i})-u_{0}\frac{\|\delta T(\frac{(1-\epsilon_{i+1})}{(1-\epsilon_{i})}x_{i})\|}{\epsilon(1-\epsilon_{0})}\\
 & \quad\geq[(1-\frac{(1-\epsilon_{i+1})}{(1-\epsilon_{i})})-\frac{\delta_{0}}{\epsilon(1-\epsilon_{0})}]u_{0}
\end{align*}
and we can finish the proof as previously for $\delta_{0}$ small
enough.\end{proof}
\begin{rem}
Actually we have that $T(x)\in\mathring{\mathcal{C}}$ for all $x\in\mathcal{C}$.
Indeed, as previously\label{rem:Open} 
\[
\frac{m\sum a_{i}}{M}u_{0}\leq T(x)
\]
and then for any $\|y\|\leq\frac{\epsilon(1-\epsilon_{0})}{M}m\sum a_{i}$
we have $T(x)+y\geq\frac{m\sum a_{i}}{M}u_{0}+y\ge0$.
\end{rem}
The construction of the cone is simple enough that we can calculate
the norm of Theorem \ref{prop:NormEqui-1}. Because of the previous
remark it is equivalent to the space norm.
\begin{prop}
There exists $c>0$ such that for any $y$ in the projected space,
\[
c\|y\|\leq\|y\|_{N}\leq\frac{1}{c}\|y\|
\]
where $\|\|_{N}$ is the norm constructed in Proposition \ref{prop:NormEqui-1}
for the cone $\mathcal{C}$ in a neighborhood of $T(x)$.\end{prop}
\begin{proof}
Because of the previous remark, $T(x)\in\mathring{\mathcal{C}}.$
Therefore there exists $r>0$ such that $B(T(x),r)\subset\mathcal{C}$.
For any $y$ we have 
\[
T(x)\left(1-\frac{s\|y\|}{r}\right)\leq T(x)+sy\leq T(x)\left(1+\frac{s\|y\|}{r}\right)
\]
and we obtain $d_{\mathcal{C}}(T(x)+sy,T(x))\leq2\frac{s\|y\|}{r}+o(\frac{s\|y\|}{r})$.
For the other direction, $\mathcal{C\subset\text{cone from }}B(u_{0},\frac{1}{2})$
(for $\epsilon$ small). In addition $T(x)+y\notin\mathcal{C}$ for
$\|y\|=1$ and then $\alpha T(x)\leq T(x)+sy$ implies $\alpha\leq(1-s)$
and so $\|y\|_{N}\leq\|y\|.$ 
\end{proof}
We can now finish the proof of Theorem \ref{Theo-smooth-quantum}
\begin{proof}
{[}Theorem \ref{Theo-smooth-quantum}{]} $T(\beta)$ is $C^{\infty}$
for the norm $\|\|$ and thanks to the previous proposition also for
any norm constructed $\||_{N}$ around $T(\beta)(x)$. We can then
apply Theorem \ref{thm:Smooth-1-1}. The analyticity follows as well:
if $T(\beta)$ is analytic with coefficient bounded by $r^{n}$ for
the norm $\|\|$ then the coefficients are bounded by $cr^{n}$ for
the norm $\mathcal{N}$.
\end{proof}
We focus now on the decay of correlation and the proof of Theorem
\ref{Theo-decay-correlation-quantum}. Recall that we want to prove
that there exists $\kappa<1$ such that 

\[
|\rho_{k}(x_{i_{1}},x_{i_{2}},\text{\dots},x_{i_{k}})-\prod\rho_{1}(x_{i_{1}})|\leq C_{k}\kappa^{\min|i_{l}-i_{l+1}|}
\]

\begin{proof}
{[}Theorem \ref{Theo-decay-correlation-quantum}{]} There exists $\epsilon_{0}>0$
such that for $f$ with $\|f\|_{L^{1}}\leq\epsilon_{0}$, $T_{i+1}(1+f)T_{i}$
are positive operator for the cone $\mathcal{C}$. 

Indeed, let $u\in\mathcal{C}$, then $T_{i}(u)\in L^{\infty}(\Gamma)$,
therefore $fT_{i}(u)\in L^{1}$ with $\|fT_{i}(u)\|_{L^{1}}\leq c\delta_{0}\|u\|$,
and finally $\|T_{i+1}fT_{i}(u)\|_{L^{\infty}}\leq c'\delta\|u\|.$
Because $T(\mathcal{C})$ is compact on the projected space and thanks
to Remark \ref{rem:Open}, there exists $r>0$ such that for all $i$,
$u\in\mathcal{C}$, $B(T_{i+1}T_{i}(u),r\|u\|)\subset\mathcal{C}$.
As a conclusion $T_{i+1}(1+f)T_{i}(u)=T_{i+1}T_{i}(u)+T_{i+1}fT_{i}(u)\in\mathcal{C}$.
We can now apply Theorem \ref{thm:TwoPointCorelation} with the same
notation. We find

\begin{align*}
 & e^{-2(k+1)R(\sum_{j=0}^{k}\kappa^{K_{j+1}-K_{j}})}\prod_{i=1}^{k}\rho_{K_{i}}(T_{l_{i}+1}(1+f_{i})T_{l_{i}})\\
 & \quad\leq\rho_{K_{k},...,K_{1}}(T_{l_{k}+1}(1+f_{k})T_{l_{k}},...,T_{l_{1}}(1+f_{1})T_{l_{1}})
\end{align*}
and 
\begin{align*}
 & \rho_{K_{k},...,K_{1}}(T_{l_{k}+1}(1+f_{k})T_{l_{k}},...,T_{l_{1}}(1+f_{1})T_{l_{1}})\\
 & \quad\leq e^{2(k+1)R(\sum_{j=0}^{k}\kappa^{K_{j+1}-K_{j}})}\prod_{i=1}^{k}\rho_{K_{i}}(T_{l_{i}+1}(1+f_{i})T_{l_{i}}).
\end{align*}
The rest follows from a induction on $k$ and this concludes the proof
of Theorem \ref{Theo-decay-correlation-quantum}.
\end{proof}
\appendix

\section{Proof of Proposition \ref{PropMarginalTronque}}

If $I'=\text{\{}i_{1},i_{2}\text{\}}$, then $x_{i_{1}}$ and $x_{i_{2}}$
are independent, $\rho_{2}(x_{i_{1}},x_{i_{2}})=\rho(x_{i_{1}})\rho(x_{i_{2}})$
and then $\rho^{T}(x_{i_{1}},x_{i_{2}})=0$. For larger a $I'$, we
have
\begin{align*}
 & \sum_{I_{1}\cup I_{2}\cup...\cup I_{r}=I',}\prod_{l=1}^{r}\rho_{|I_{l}|}^{T}((x_{i})_{i\in I_{l}})\\
 & =\sum_{\substack{I_{1}\cup I_{2}\cup...\cup I_{r}=\text{\{}1,...,n\text{\}}\\
I_{l}\subset I\text{ or }I_{l}\subset J.
}
}\prod_{l=1}^{r}\rho_{|I_{l}|}^{T}((x_{i})_{i\in I_{l}})\\
 & =\left(\sum_{\substack{I_{1}\cup I_{2}\cup...\cup I_{r_{1}}=I\cap I'}
}\prod_{l=1}^{r_{1}}\rho_{|I_{l}|}^{T}((x_{i})_{i\in I_{l}})\right)\left(\sum_{\substack{I_{1}\cup I_{2}\cup...\cup I_{r_{2}}=J}
\cap I'}\prod_{l=1}^{r_{1}}\rho_{|I_{l}|}^{T}((x_{i})_{i\in I_{l}})\right)\\
 & =\rho_{|I'\cap I|}((x_{i})_{i\in I\cap I'})\rho_{|I'\cap J|}((x_{i})_{i\in J\cap I'})
\end{align*}
and therefore
\begin{align*}
 & \rho_{|I'|}^{T}((x_{i})_{i\in I})\\
 & =\rho_{|I'\cap I|}((x_{i})_{i\in I\cap I'})\rho_{|I'\cap J|}((x_{i})_{i\in J\cap I'})-\sum_{I_{1}\cup I_{2}\cup...\cup I_{r}=I',}\prod_{l=1}^{r}\rho_{|I_{l}|}^{T}((x_{i})_{i\in I_{l}})\\
 & =0.
\end{align*}

\bibliographystyle{plain}
\bibliography{BiblioLocalisationAnderson}

\begin{thebibliography}{10}

\bibitem{AizJanJun-10}
Michael Aizenman, Sabine Jansen, and Paul Jung.
\newblock {Symmetry Breaking in Quasi-1D Coulomb Systems}.
\newblock {\em Ann. Henri Poincar\'e}, 11(8):1453--1485, Dec 2010.

\bibitem{aizenman1980structure}
Michael Aizenman and Philippe~A Martin.
\newblock Structure of {G}ibbs states of one dimensional {C}oulomb systems.
\newblock {\em Communications in Mathematical Physics}, 78(1):99--116, 1980.

\bibitem{amann1976fixed}
Herbert Amann.
\newblock Fixed point equations and nonlinear eigenvalue problems in ordered
  {B}anach spaces.
\newblock {\em SIAM review}, 18(4):620--709, 1976.

\bibitem{Baxter-63}
R.~J. Baxter.
\newblock Statistical mechanics of a one-dimensional {C}oulomb system with a
  uniform charge background.
\newblock {\em Proc. Cambridge Philos. Soc.}, 59:779--787, 1963.

\bibitem{birkhoff1957extensions}
Garrett Birkhoff.
\newblock Extensions of {J}entzsch's theorem.
\newblock {\em Transactions of the American Mathematical Society},
  85(1):219--227, 1957.

\bibitem{blanc2015crystallization}
Xavier Blanc and Mathieu Lewin.
\newblock {The Crystallization Conjecture: A Review}.
\newblock {\em EMS Surv. Math. Sci.}, 2(2):219--306, 2015.

\bibitem{borwein1994entropy}
Jonathan~M Borwein, Adrian~Stephen Lewis, and Roger~D Nussbaum.
\newblock Entropy minimization, dad problems, and doubly stochastic kernels.
\newblock {\em Journal of Functional Analysis}, 123(2):264--307, 1994.

\bibitem{brack1993physics}
Matthias Brack.
\newblock The physics of simple metal clusters: self-consistent {J}ellium model
  and semiclassical approaches.
\newblock {\em Reviews of modern physics}, 65(3):677, 1993.

\bibitem{brascamp2002some}
HJ~Brascamp and EH~Lieb.
\newblock Some inequalities for {G}aussian measures and the long-range order of
  the one-dimensional plasma.
\newblock In {\em Inequalities}, pages 403--416. Springer, 2002.

\bibitem{Choquard-75}
Ph. Choquard.
\newblock On the statistical mechanics of one-dimensional {C}oulomb systems.
\newblock {\em Helv. Phys. Acta}, 48(4):585--598, 1975.

\bibitem{cohen1979ergodic}
Joel~E Cohen.
\newblock Ergodic theorems in demography.
\newblock {\em Bulletin (New Series) of the American Mathematical Society},
  1(2):275--295, 1979.

\bibitem{ErdosUniversality}
L\'aszl\'o Erd\H{o}s.
\newblock Universality for random matrices and log-gases.
\newblock In {\em Current developments in mathematics 2012}, pages 59--132.
  Int. Press, Somerville, MA, 2013.

\bibitem{eveson1995elementary}
Simon~P Eveson and Roger~D Nussbaum.
\newblock An elementary proof of the birkhoff-hopf theorem.
\newblock In {\em Mathematical Proceedings of the Cambridge Philosophical
  Society}, volume 117, pages 31--55. Cambridge University Press, 1995.

\bibitem{forrester2010log}
Peter~J Forrester.
\newblock {\em Log-gases and random matrices (LMS-34)}.
\newblock Princeton University Press, 2010.

\bibitem{GiuVig-05}
G.~Giuliani and G.~Vignale.
\newblock {\em Quantum Theory of the Electron Liquid}.
\newblock Cambridge University Press, 2005.

\bibitem{GruLebMar-81}
Ch. Gruber, {Joel L.} Lebowitz, and {Ph. A.} Martin.
\newblock Sum rules for inhomogeneous {C}oulomb systems.
\newblock {\em J. Chem. Phys.}, 75(2):944--954, 1981.

\bibitem{GruLugMar-78}
Ch~Gruber, Ch~Lugrin, and Ph~A Martin.
\newblock Equilibrium equations for classical systems with long range forces
  and application to the one dimensional {C}oulomb gas.
\newblock {\em Helv Phys. Acta}, 51(5-6):829--866, 1978.

\bibitem{guivarc1988theoremes}
Yves Guivarc'h and Jean Hardy.
\newblock Th{\'e}or{\`e}mes limites pour une classe de cha{\^\i}nes de {M}arkov
  et applications aux diff{\'e}omorphismes d'{A}nosov.
\newblock 24(1):73--98, 1988.

\bibitem{hall2014martingale}
Peter Hall and Christopher~C Heyde.
\newblock {\em Martingale limit theory and its application}.
\newblock Academic press, 2014.

\bibitem{herve2008vitesse}
Lo\"{i}c Herv\'e.
\newblock Vitesse de convergence dans le th\'eor\`eme limite central pour des
  cha\^\i nes de {M}arkov fortement ergodiques.
\newblock {\em Ann. Inst. Henri Poincar\'e Probab. Stat.}, 44(2):280--292,
  2008.

\bibitem{HohKoh-64}
P.~Hohenberg and W.~Kohn.
\newblock Inhomogeneous electron gas.
\newblock {\em Phys. Rev.}, 136(3B):B864--B871, Nov 1964.

\bibitem{hopf1963inequality}
Eberhard Hopf.
\newblock An inequality for positive linear integral operators.
\newblock {\em J. Math. Mech.}, 12:683--692, 1963.

\bibitem{JanJun-14}
S.~Jansen and P.~Jung.
\newblock Wigner crystallization in the quantum 1d jellium at all densities.
\newblock {\em Comm. Math. Phys.}, pages 1--22, 2014.

\bibitem{JanLieSei-09}
S.~Jansen, E.~H. Lieb, and R.~Seiler.
\newblock Symmetry breaking in {L}aughlin's state on a cylinder.
\newblock {\em Comm. Math. Phys.}, 285(2):503--535, Jan 2009.

\bibitem{jansen2014wigner}
Sabine Jansen and Paul Jung.
\newblock Wigner crystallization in the quantum 1d {J}ellium at all densities.
\newblock {\em Communications in Mathematical Physics}, 331(3):1133--1154,
  2014.

\bibitem{KUNZ1974303}
H~Kunz.
\newblock The one-dimensional classical electron gas.
\newblock {\em Annals of Physics}, 85(2):303 -- 335, 1974.

\bibitem{le1982theoremes}
{\'E}mile Le~Page.
\newblock Th{\'e}oremes limites pour les produits de matrices al{\'e}atoires.
\newblock In {\em Probability measures on groups}, pages 258--303. Springer,
  1982.

\bibitem{lewin2017statistical}
Mathieu Lewin, Elliott~H. Lieb, and Robert Seiringer.
\newblock {Statistical Mechanics of the Uniform Electron Gas}.
\newblock {\em J. \'Ec. polytech. Math.}, 5:79--116, 2018.

\bibitem{LunMar-83}
S.~Lundqvist and N.H. March, editors.
\newblock {\em {Theory of the Inhomogeneous Electron Gas}}.
\newblock Physics of Solids and Liquids. Springer US, 1983.

\bibitem{ParYan-94}
R.G. Parr and W.~Yang.
\newblock {\em Density-Functional Theory of Atoms and Molecules}.
\newblock International Series of Monographs on Chemistry. Oxford University
  Press, USA, 1994.

\bibitem{pfeuty1970one}
Pierre Pfeuty.
\newblock The one-dimensional {I}sing model with a transverse field.
\newblock {\em Annals of Physics}, 57(1):79--90, 1970.

\bibitem{ruelle1999statistical}
David Ruelle.
\newblock {\em Statistical mechanics: Rigorous results}.
\newblock World Scientific, 1999.

\bibitem{wigner1934interaction}
Eugene Wigner.
\newblock On the interaction of electrons in metals.
\newblock {\em Physical Review}, 46(11):1002, 1934.

\end{thebibliography}

\end{document}